\documentclass[10pt,a4paper,final]{article}
\usepackage[utf8]{inputenc}
\usepackage{amsmath}
\usepackage{amsfonts}
\usepackage{amssymb}
\usepackage{graphicx}
\usepackage{mathabx}
\usepackage{mathtools}
\usepackage{breqn}
\usepackage{mathrsfs}
\usepackage{bm}
\usepackage{amsmath,amsfonts,amssymb,amsthm,epsfig,epstopdf,titling,url,array}
\usepackage[font=small,labelfont=bf]{caption}
\usepackage{hyperref}
\usepackage{subcaption}
\usepackage{authblk}

\newtheorem{mydef}{Definition}[subsection]
\newtheorem{prop}{Proposition}[subsection]
\newtheorem*{remark}{Remark}
\newtheorem{thm}{Theorem}[subsection]
\newtheorem{lem}{Lemma}[subsection]
\newtheorem{cor}{Corollary}[subsection]
\newtheorem{ex}{Example}[subsection]

\newcommand{\booktitle}[1]{\textit{#1}}

\DeclareMathOperator*{\im}{im}
\DeclareMathOperator*{\rk}{rk}

\title{A combinatorial interpretation of harmonic cycles}
\author[1]{Younng-Jin Kim \thanks{sptz@snu.ac.kr}}
\affil[1]{Department of Mathematical Sciences, Seoul National University}

\begin{document}
\maketitle
 
\begin{center}
Abstract.
\end{center}

   In this paper, we will investigate a harmonic cycle (discrete harmonic form). With a CW-complex $X$, we can construct the combinatorial Laplacian operator  $\bigtriangleup_*=\partial_{*+1} \partial^t_{*+1}+\partial^t_* \partial_*:C_*\rightarrow C_*$. The kernel $\mathcal{H}_*(X)=\ker \bigtriangleup_*$ is the harmonic space, the set of harmonic cycles, and is isomorphic to its homology due to combinatorial Hodge theory.

We will introduce four concepts; cycletree, unicyclization, winding number map, and standard harmonic cycle. A cycletree is $T\coprod e$ as a combinatorial object where $T$ is a spanning tree and $e$ is an edge on a given graph structure $G$. A unicyclization $\mathscr{A}$ consists of $G$ and information $\partial$, a substitute for faces. Then, we will define the winding number map $w_{\mathscr{A}}$ and the standard harmonic cycle $\lambda_{\mathscr{A}}$ and prove related properties.

Finally, we will show a relation between the winding number and the inner product with the standard harmonic cycle. Then, we will see the standard harmonic cycle is actually a harmonic cycle. In other words, we give a combinatorial interpretation on a harmonic cycle.

\tableofcontents

\section{Introduction}
A Harmonic cycle is the element in the kernel of Combinatorial Laplacian. Due to combinatorial Hodge decomposition, we can get the isomorphism between homology and the set of harmonic cycles for a given CW-complex $X$. In other words, there is a specific relation between homology and solutions in numerical PDE equations of discrete settings. We give a combinatorial meaning on harmonic cycles. 

For important related work, there are papers like \cite{NS} and \cite{Ca}. These papers prove similar results by using projections. 

In this paper, we've dealt with harmonic cycle with winding number even for a harmonic cycle of multi rank case. Moreover, we have studied the relation of harmonic cycles between similar unicyclizations or CW-complexes.

\section{Preliminaries}

\subsection{Review of finite chain complexes}
In this paper, we will assume basic knowledge of finite chain complexes and homology groups with coefficients in $\mathbb{Z}$ or $\mathbb{R}$.  Familiarity with cell complexes will be helpful. One may refer to standard texts such as  \cite{Mu} for details. 

For $d>0$, let $X=X_{0}\amalg \cdots \amalg X_{d}$ where $X_{i}$ is a nonempty finite set for each $i\in[0,d]$. We will refer to $X$ as a \emph{complex} of dimension $d$. (For example, $X$ may be a cell complex of dimension $d$ with $X_{i}$ the set of $i$-dimensional cells in $X$.) 
%The $i$-skeleton $X^{(i)}$ of $X$ is $X_{0}\cup X_{1}\cup\cdots \cup X_{i}$.   
%This condition on $X$ allows one to represent the boundary maps $\partial_{i}$ of its chain complex as matrices. 
%Also we define $X_{-1}$ to be a set with one element.
%where we define $X_{-1}=\{\emptyset\}$.
%Also we define each of $X_{-2}$ and $X^{(-2)}$ to be the void set. 
The $i$-th chain group of $X$ is the free abelian group  $C_{i}=C_{i}(X)\cong \mathbb{Z}^{|X_{i}|}$ generated by $X_{i}$.  The $i$-th boundary map $\partial_{i}=\partial_{i}(X):C_{i}\rightarrow C_{i-1}$ is an integer matrix whose rows and columns are indexed by $X_{i-1}$ and $X_{i}$, respectively, satisfying $\partial_{i-1}\partial_{i}=0$ for all $i$.  We define $\partial_{i}=0$ for $i\notin [1,d]$. The $i$-th \emph{coboundary map} $\partial^{t}_{i}:C_{i-1}\rightarrow C_{i}$ is the transpose of $\partial_{i}$. We will often refer to $X$ as the generator of the chain complex $\{C_{i},\partial_{i}\}_{i\in[0,d]}$.
%with the augmentation $\partial_{0}:C_{0}\rightarrow C_{-1}\cong\mathbb{Z}$ given by $\partial_{0}(v)=1$ for every $v\in X_{0}$.  

The $i$-th homology and cohomology groups of $X$ are defined $H_{i}(X)=\ker\partial_{i}/\im\,\partial_{i+1}$ and $H^{i}(X)=\ker\partial^{t}_{i+1}/\im\,\partial^{t}_{i}$, respectively. The elements of $Z_{i}:=\ker\partial_{i}$ and $Z^{i}:=\ker\partial^{t}_{i+1}$ are called $i$-cycles and $i$-cocycles, respectively. The $i$-th chain group and (co)homology group of $X$ with $\mathbb{R}$-coefficients are denoted $C_{i}(X;\mathbb{R})$ and $H_{i}(X;\mathbb{R})$. Note that $H_{i}(X;\mathbb{R})\cong H^{i}(X;\mathbb{R})$ for all $i$ as $\mathbb{R}$-vector spaces.
%$X$ is {\it acyclic} if $\tilde{H}_{i}(X)=0$ for all $i\in[-1,d]$. 
%Define $C_{-1}=\mathbb{Z}$ and $C_{i}=0$ for $i\notin [-1,d]$.
%and $C_{i}=0$ for $i\notin[-1,d]$. 
%Define $\tilde{H}_{i}(X)=0$ for $i\leq -1$.
%Note that $\tilde{H}_{d}(X)=\mbox{Ker}\,\partial_{d}$ is free abelian.
%Recall that $\mathrm{rk}\,\tilde{H}_{i}(X)=0$ iff $\tilde{H}_{i}(X)$ is finite.

%$$\Delta_{-1}=\partial_{0}\partial_{0}^{t}\, , \mbox{ and } \Delta_{d}=\partial^{t}_{d}\partial_{d}\, .$$   
%Since $L_{i}:=\partial_{i+1}\partial^{t}_{i+1}$ and $J_{i}:=\partial^{t}_{i}\partial_{i}$ are symmetric, non-negative definite, and $L_{i}J_{i}=J_{i}L_{i}=0$ because $\partial_{i}\partial_{i+1}=0$. Hence, each $\Delta_{i}$ is also symmetric and non-negative definite by Lemma \ref{matrixtheory1}.
%Suppose that $\partial_{i}\not=0$ for all $i\in [0,d]$. 
%(This proposition does not assume $X_{i}\not=\emptyset$ for every $i\in[0,d]$). 

\subsection{Harmonic cycles and combinatorial Hodge theory} In this subsection, we will define harmonic cycles, a main object of study of this paper, via combinatorial Laplacians, and recall combinatorial Hodge theory relating harmonic cycles and homology groups. 

Given a chain complex 
$\{C_{i}(X;\mathbb{R}),\partial_{i}\}_{i\in [0,d]}$ generated by $X=X_{0}\amalg\cdots\amalg X_{d}$,      
the $i$-th combinatorial Laplacian $\Delta_{i}=\Delta_{X,i}:C_{i}(X;\mathbb{R})\rightarrow C_{i}(X;\mathbb{R})$ is 
$$\Delta_{i}=\partial_{i}^t \partial_{i} + \partial_{i+1} \partial^t_{i+1}$$
for $i\in[0,d]$ (refer to \cite{E}).  
The $i$-th \emph{harmonic space} $\mathcal{H}_i(X)$ of the chain complex is $\ker\Delta_{i}$, and  
its elements are called $i$-\emph{harmonic cycles}.

Regard $C_{i}(X;\mathbb{R})$ as an 
$\mathbb{R}$-vector space endowed with a standard inner product $\circ$ such that the set $X_{i}$ of its generators forms an orthonormal basis. From the orthogonal decomposition $C_{i}(X;\mathbb{R})=\mathcal{H}_{i}(X)\oplus\im\partial_{i}^{t}\oplus\im\partial_{i+1}$ (refer to \cite{Fr}),
one can deduce 
\begin{equation}\label{harmonics}
\mathcal{H}_i(X)=\ker\partial_{i} \cap \ker\partial^t_{i+1}
\end{equation}
via $(\im M^{t})^{\perp}=\ker M$ for a matrix $M$.  Hence, we have an important fact that \emph{a harmonic cycle is both a cycle and a cocycle}.

From the above orthogonal decomposition for $C_{i}(X;\mathbb{R})$, we also see that $\ker\partial_{i}=\mathcal{H}_{i}(X)\oplus\im\partial_{i+1}$. Hence we have the following main statement of combinatorial Hodge theory:  for each $i\in[0,d]$
\begin{equation}\label{hodge}
\mathcal{H}_i(X)\cong H_{i}(X;\mathbb{R})
\end{equation}
as $\mathbb{R}$-vector spaces. In particular $\rk\mathcal{H}_{i}(X)=\rk H_{i}(X;\mathbb{R})$ for all $i$. (One can show that this isomorphism maps a harmonic cycle $h$ to its homology class $\overline{h}$.) 
%An important property of a harmonic class is energy minimizing 

The following \emph{energy minimizing property} of a harmonic class is a consequence of (\ref{hodge}): For $h\in \mathcal{H}_{i}(X)$ and $x\in \overline{h}$, 
\begin{equation}\label{energyminimizing}
h\circ h\leq x\circ x.
\end{equation} 
Indeed, this inequality follows easily from the facts $x=h+\partial_{i+1}y$ for some $y\in C_{i+1}(X;\mathbb{R})$ and $h\circ (\partial_{i+1}y)=(\partial_{i+1}^{t}h)\circ y=0$ because $h\in\ker\partial^t_{i+1}$. 

%
%.
%
%
%Since $A_i(X)$ (or $A_i(X; \mathbb{R})$) is a free module with the $i$-cell basis, so it has the natural inner product $\circ$. It means that we have the $l^2$-norm in $A_i$. For example, here is the following remark.

%\begin{remark}let $P$ be a path (or walk) in a graph $G=X^1$ of a CW-complex $X$. We may understand $P$ as an element of $A_1$ (i.e., a vector whose non-zero component is $\pm 1$). For any $\alpha \in A_1$, the inner product between $P$ and $\alpha$ can be written as
%\begin{equation}P\circ \alpha=\sum_i (the\ orientation\ of\ p_i)\times (the\ value\ of\ \alpha\ on\ edge\ p_i)\end{equation}
%Note that the previous equation is independent of how the (directed) edges of the graph $G$ are oriented.\end{remark}

\begin{ex}
\ \\
\begin{center}
\includegraphics[scale=0.3]{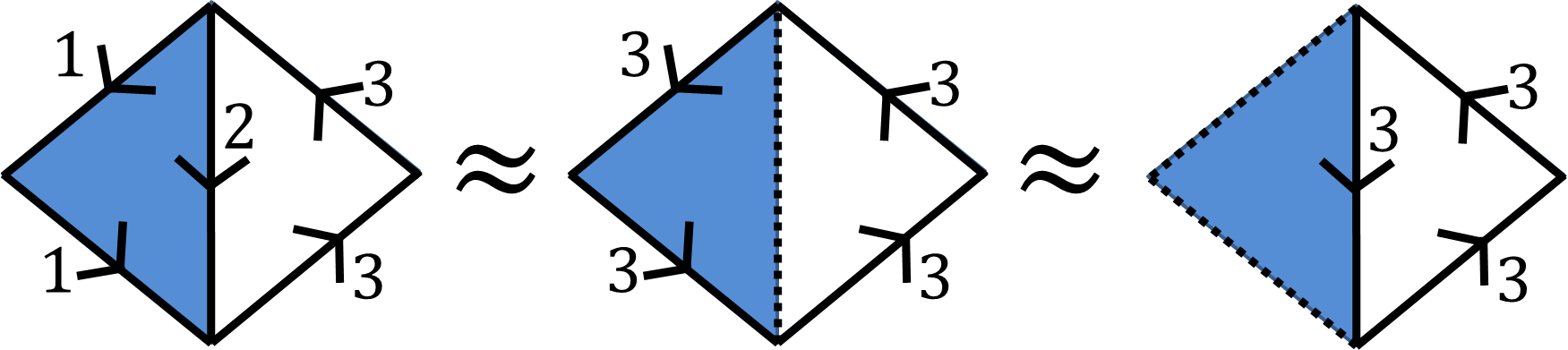}
\end{center}

The three cycles are homologous as elements of homology at dimension 1. However, their norms are different. To be specific, they are 1+1+4+9+9=24, 9+9+9+9=36, and 9+9+9=27 form left to right. Actually, the left cycle is a harmonic class, and its norm is the smallest norm among its homologous elements.
\end{ex}

\begin{thm}
(Mean value property at dimension 0) Let's assume $X$ is a locally finite CW-complex. Then, $h \in \mathcal{H}_0(X)$ satisfying the following mean value property: the coefficient of $h$ at a vertex $v$ is the mean of coefficients of $h$ at nearby vertices of $v$.
%- the value of $h$ at a n-simplex $u$ is the mean of $h\circ$[n-dimensional flow around $u$] over (n+1)-simplex $\sigma$ which has $u$ as its face, where a n-dimensional flow around $u$ is $\partial_{n+1} (\pm \sigma) - u$ as an element of $C_1$ and the sign $\pm$ is chosen so that $\partial_{n+1} (\pm \sigma) - u$ is zero at $u$.

%- the value of $h$ at a n-simplex $u$ is the sum of $h\circ$[n-dimensional flow near $u$] over n-simplex $\sigma$ which share the common face with $u$, where a n-dimensional flow near $u$ is $\partial_{n} (\pm \sigma) - u$ as an element of $C_1$ and the sign $\pm$ is chosen so that $\partial_{n+1} (\pm \sigma) - u$ is zero at $u$.
\end{thm}
%Of course, we can check mean value property with a CW-complex X, but it has a messy interpretation.\\

\begin{proof}
Since $\partial_0=0$, $h \in \mathcal{H}_0(X)$ is equivalent to $h\in \ker{\partial^t_{1}}$. Note that $\ker{\partial^t_{1}}=\ker{\partial_{1}\partial^t_{1}}$. The matrix ${\partial_{1}\partial^t_{1}}$
is the number of adjacent vertices of $v$ at the entry $(v,v)$, $-1$ at the entry $(v,u)$ where $v, u$ are distinct vertices, and 0 at the other entries. Therefore, by writing down the equation $(\partial_{1}\partial^t_{1})h=0$, we get the mean value property.
\end{proof}

\begin{ex}
On a $\mathbb{Z} \times \mathbb{Z}$ mesh, we may consider its harmonic cycle of dimension 0. As in the following figure, 0-dimensional harmonic cycle is a vector with respect to the vertex basis of this mesh. We may observe the mean value property holds.
\begin{center}
\includegraphics[scale=0.3]{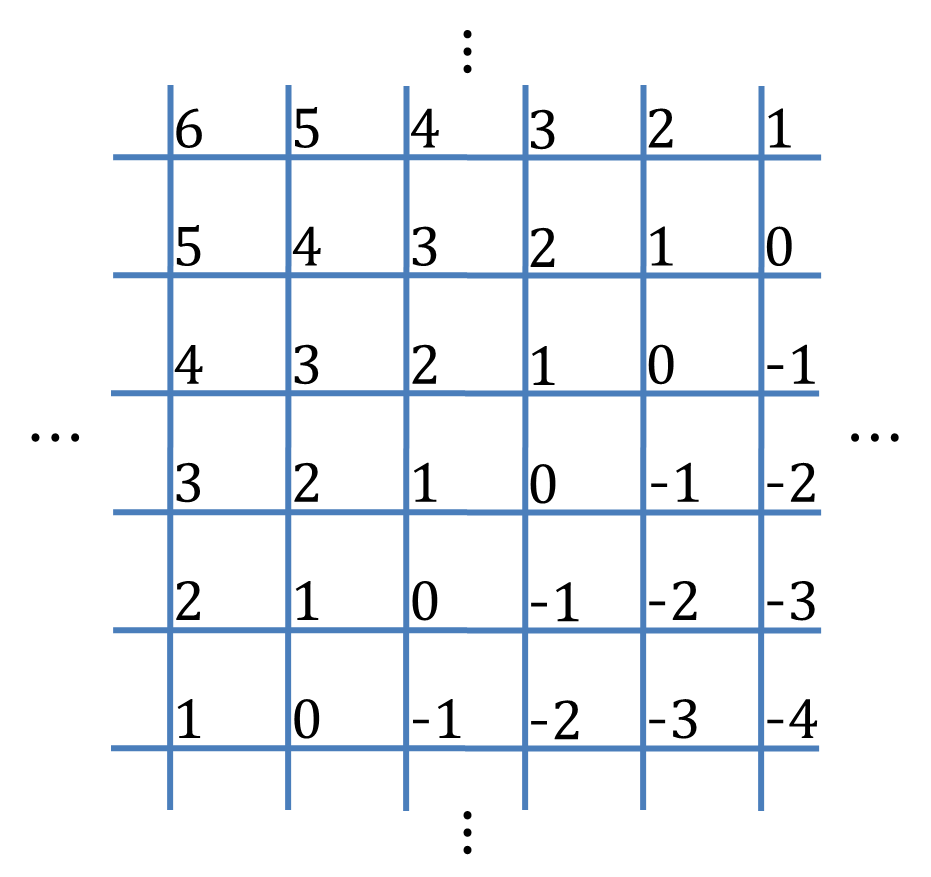}
\end{center}
\end{ex}
\begin{remark}
By reading the entries of $\bigtriangleup_i$, we can get the analogous mean value property at high dimensions.
\end{remark}

\section{Spanning trees and Cycletrees}
We refer the readers to \cite{BM} for basic definitions concerning graphs. 
In this paper, we assume that a graph $G$ is connected with loops and multiple edges allowed. Also we assume that its vertex set $V(G)$ has $n$ elements, and its edge set $E(G)$ is a multiset.  

\subsection{Spanning trees of a graph}
A subgraph of $G$ is \emph{spanning} if the vertex set of the subgraph equals $V(G)$. A \emph{spanning tree} $T$ in a connected graph $G$ is a spanning subgraph which is connected and has no cycle. One can show that every spanning tree has $n-1$ edges. 
%and we call $n-1$ the \emph{rank} and $c=m-(n-1)$ the \emph{corank} of $G$.  
We will denote the collection of all spanning trees in $G$ by $\mathcal{T}(G)$ and call the number of spanning trees in $G$ the \emph{tree-number} of $G$,  denoted by $k(G)$.

Given a graph $G$ and a non-loop edge $\sigma\in E(G)$, the \emph{contraction} $G/\sigma$ is a graph obtained by contracting $\sigma$, i.e., by removing $\sigma$ from $G$ and identifying its end vertices. In particular, we have $|V(G/\sigma)|=|V(G)|-1$ and $|E(G/\sigma)|=|E(G)|-1$. The \emph{deletion} $G-\sigma$ is obtained by deleting $\sigma$ from $G$. Hence, $|V(G-\sigma)|=|V(G)|$ and $|E(G-\sigma)|=|E(G)|-1$. If $\sigma$ is a loop, we define $G/\sigma=G-\sigma$ with no changes in the vertices. 

Let $\sigma\in E(G)$. If $\sigma$ is non-loop, then the subset of $\mathcal{T}(G)$ consisting of all spanning trees in $G$ that contain $\sigma$ corresponds bijectively to the set $\mathcal{T}(G/\sigma)$ of all spanning trees in $G/\sigma$, where the correspondence is given by the contraction of $\sigma$. If $\sigma$ is a loop, then $\mathcal{T}(G)$ corresponds bijectively to both $\mathcal{T}(G/\sigma)$ and $\mathcal{T}(G-\sigma)$. Hence, we have  
%Recall that for a non-loop edge $\sigma\in E(G)$, the deletion-contraction recurrence for the number $k(G)$ of the spanning trees in $G$ is 
%\begin{remark} 
\begin{equation}\label{del-con:T(G)}
\begin{array}{ll}
\mathcal{T}(G) \xleftrightarrow[]{bij.} \mathcal{T}(G-\sigma)\amalg \mathcal{T}(G/\sigma)
 & \condition[]{if $\sigma$ is not a loop}. \\
 \mathcal{T}(G/\sigma) \xleftrightarrow[]{bij.} \mathcal{T}(G) \xleftrightarrow[]{bij.} \mathcal{T}(G-\sigma)
 & \condition[]{if $\sigma$ is a loop}. 
\end{array}
\end{equation}
For the first case, we get
\begin{equation}\label{del-con:k(G)}
k(G)=k(G-\sigma)+k(G/\sigma).
\end{equation}

\subsection{Cycletrees of a graph}

%A cycletree is not a cycle, nor a tree, but it is a notion related to them. 

\begin{mydef}
Let $G$ be a connected graph. A cycletree $Y$ in $G$ is a connected spanning subgraph of $G$ with exactly one cycle. $\mathcal{U}(G)$ will denote the set of all cycletrees in G. 
\end{mydef}
Note that a cycletree $Y\in \mathcal{U}(G)$ can be expressed as a union  
\begin{equation}
Y=T\cup \{\sigma\}
\end{equation} 
of a spanning tree $T$  in $G$ and an edge $\sigma \in E(G)-E(T)$. Hence, a cycletree $Y\in \mathcal{U}(G)$ is a connected spanning subgraph of $G$ satisfying 
\begin{equation}\label{CTEuler}
|E(Y)|=|V(G)|.
\end{equation} 

For a cell complex $X$, a cycletree of $X$ will mean a cycletree of its 1-skeleton $X^{(1)}=X_{0}\cup X_{1}$, and $\mathcal{U}(X)$ will mean $\mathcal{U}(X^{(1)})$. In other literature, a cycletree is also called as a cycle-rooted spanning tree \cite{Ke}, or some co-tree \cite{Ca}. 

\begin{ex}
For a left given graph $G$, there is the list of cycletrees as a green graph. Also, for each cycletree, we can see the unique cycle of it as a red graph.
\begin{center}
\includegraphics[scale=0.3]{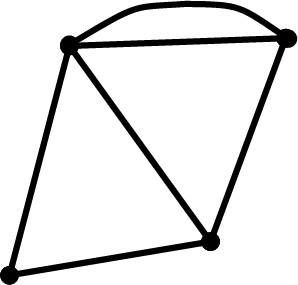}
\ \ \ \ \ \ \ \
\includegraphics[scale=0.15]{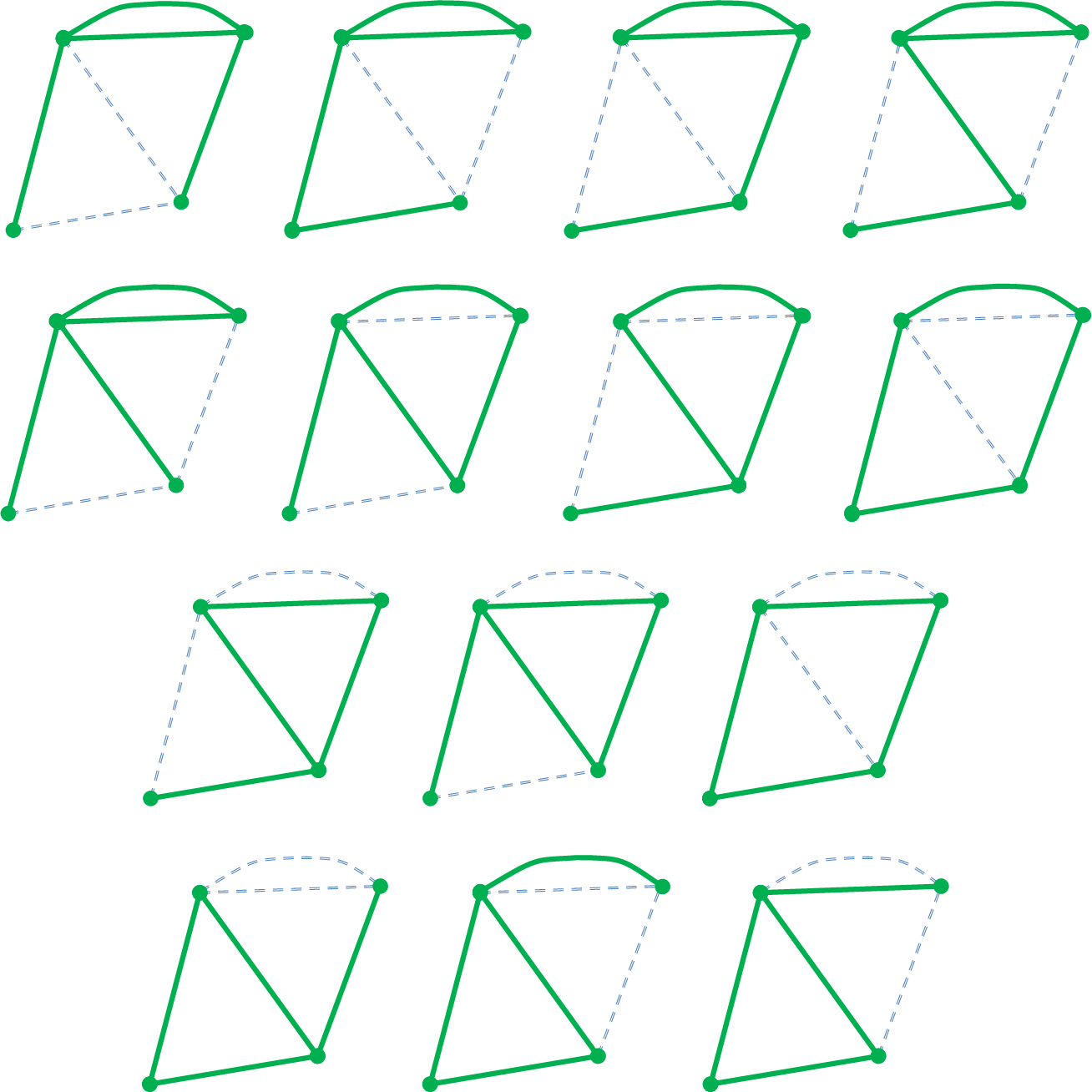}
\ \ \ \ \ \ \ \
\includegraphics[scale=0.15]{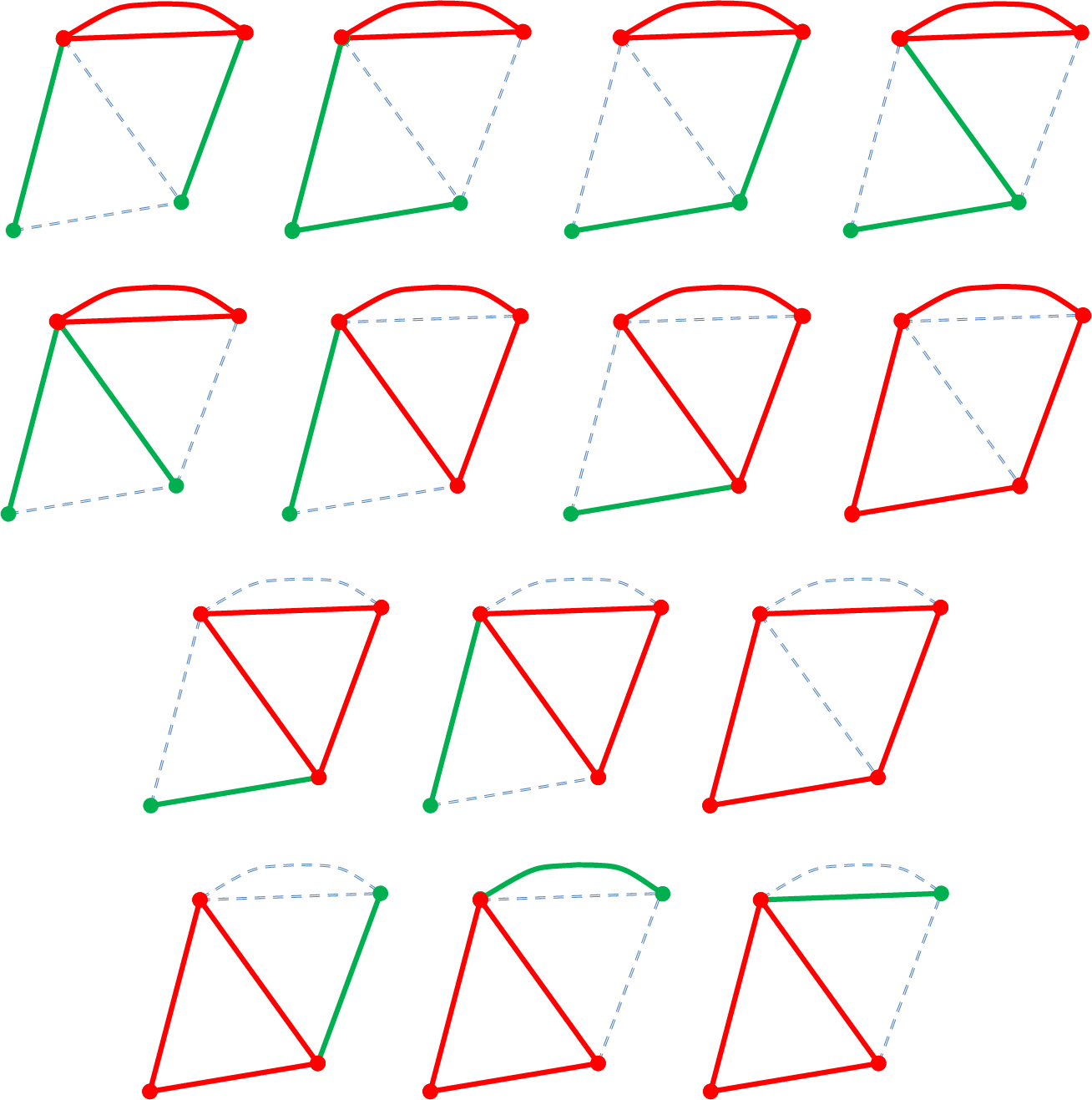}
\end{center}
\end{ex}

\

%\begin{mydef}
%Let $X$ be a CW-complex. A cycletree of $X$ is the cycletree of $X^1$ (i.e., $\mathcal{U}(X)\equiv \mathcal{U}(X^1)$).
%\end{mydef}
%We may consider a circuit, or closed path, as a cycle. In other words, a circuit can be represented as a vector in $Z_1(X)\leqslant A_1(X)$ composed of $0$ and $\pm 1$ entries. 

We will denote the unique cycle in a cycletree $Y\in\mathcal{U}(G)$ by $C_{Y}$ and its corresponding element in $Z_{1}(G)$ by $z_{Y}$.  Here, we assume that an orientation of $C_{Y}$ is fixed so that $z_{Y}$ is well defined. As we shall see, our main results are independent of these orientations.   

%\begin{equation}
%\mathcal{U}_{o}(X)=\{(Y,C_Y)|\ Y\ is\ a\ cycletree\ in\ X\}.
%\end{equation}
%Let's call $\mathcal{U}_{o}(X)$ as an oriented cycletree set. Sometimes, write $\mathcal{U}_{o}(X)=\{(Y,C_Y)\}$ for convenience.
%Note that the circuit $C_Y$ in $\mathcal{U}_{o}(X)(=\{(Y,C_Y)\})$ is usually used as a vector in $Z_1(X)$ in this paper.

\subsection{Deletion and contraction for cycletrees}

%With respect to cycletree, we can get another the deletion-contraction bijection between graphs. The original bijection is stated in the remark after the following proposition. Here, edge-deletion means a chosen edge is removed from a graph, while its endpoints remain on that graph. Edge-contraction means a chosen edge shrinks to be a vertex with its endpoints. 

\begin{lem}\label{CTcontraction} For a connected graph $G$  and $\sigma\in E(G)$,  
let $\mathcal{U}(G)_{\sigma}$ denote the set $\{Y\in \mathcal{U}(G)\mid \sigma\in Y\}$. If $\sigma$ is not a loop, the map from $\mathcal{U}(G)_{\sigma}$ to $\mathcal{U}(G/\sigma)$ given by $Y\mapsto Y/\sigma$ is a bijection. 
If $\sigma$ is a loop, the map from $\mathcal{U}(G)_{\sigma}$ to $\mathcal{T}(G/\sigma)$ given by $Y\mapsto Y-\sigma$ is a bijection. 
\end{lem}
\begin{proof} Note that there is a natural bijective map from $E(G)\setminus \sigma$ to $E(G/\sigma)$. This map induces a bijection $f$ from the set of all spanning subgraphs of $G$ containing $\sigma$ to the set of all spanning subgraphs of $G/\sigma$, where $f$ is given by contracting the edge $\sigma$. We shall see that the restriction of $f$ to $\mathcal{U}(G)_{\sigma}$ gives the desired bijections of the proposition. 
Recall that $Y\in \mathcal{U}(G)_{\sigma}$ iff $Y$ is a connected spanning subgraph of $G$ containing $\sigma$ such that $|E(Y)|=|V(G)|$.   

Suppose $\sigma$ is not a loop. Then, we have $|E(Y/\sigma)|=|E(Y)|-1=|V(G)|-1=|V(G/\sigma)|$, and $Y/\sigma$ is a connected spanning subgraph of $G/\sigma$. Therefore we conclude $Y/\sigma\in \mathcal{U}(G/\sigma)$.  Conversely, let $Y'\in \mathcal{U}(G/\sigma)$. Then, by definition, we have $Y'=Y/\sigma$ for some spanning subgraph $Y$ of $G$ containing $\sigma$. Further, $Y$ must be connected because otherwise $Y/\sigma$ would not be connected, a contradiction.  Since we have $|E(Y)|=|E(Y')|+1=|V(G/\sigma)|+1=|V(G)|$, we conclude $Y\in \mathcal{U}(G)_{\sigma}$, which proves the first statement of the lemma. 

Suppose $\sigma$ is a loop. Then, a cycle tree $Y\in \mathcal{U}(G)_{\sigma}$ is precisely of the form $Y=T\cup\sigma$ for some spanning tree $T\in \mathcal{T}(G)$. Hence, the map $Y\mapsto Y-\sigma=Y/\sigma$ gives a bijection from $Y\in \mathcal{U}(G)_{\sigma}$ to $\mathcal{T}(G-\sigma)=\mathcal{T}(G/\sigma)$, which proves the second statement of the lemma.   
\end{proof}

\begin{prop}% ( $\mathcal{U}$ - $\mathcal{T}$ bijection) 
%We can construct $G_{c} \equiv X/\sigma$ (contraction), $G_{d}\equiv X-\sigma$ (deletion). 
For a connected graph $G$ and $\sigma\in E(G)$, we have the following natural bijections: 
%from deletion and contraction with respect to $\sigma$.
\begin{equation}
\begin{array}{ll}
\mathcal{U}(G) \xleftrightarrow[]{bij.} \mathcal{U}(G/\sigma)\amalg \mathcal{U}(G-\sigma)
 & \condition[]{if $\sigma$ is not a loop}, \\
\mathcal{U}(G) \xleftrightarrow[]{bij.} \mathcal{T}(G/\sigma)\amalg \mathcal{U}(G-\sigma)
 & \condition[]{if $\sigma$ is a loop}. 
\end{array}
\end{equation}
\end{prop}
\begin{proof}
From $\mathcal{U}(G)=\{Y\in \mathcal{U}(G)\mid \sigma\in Y\}\amalg\{Y\in \mathcal{U}(G)\mid \sigma\notin Y\}$, the proposition is clear by Lemma \ref{CTcontraction}. 
%Send the elements in $G$ to $G_{c}$ after contracting $\sigma$, if it contains $\sigma$. 
\end{proof}

\section{Winding number and harmonic cycle}
If $X=X_{0}\amalg X_{1}\amalg X_{2}$ has the property $\rk H_{1}(X)=1$, then we will call $X$ \emph{unicyclic}. For example, a connected planar graph with all of its finite faces ``filled in''  except one is unicyclic. If $X$ is unicyclic,  then we have $\rk \mathcal{H}_{1}(X)=1$ by combinatorial Hodge theory, i.e., there is a nonzero harmonic cycle $\lambda$ that spans $\mathcal{H}_{1}(X)$. The purpose of this section is to give an explicit description of $\lambda$ via cycletrees with the winding numbers of them when a connected graph is unicyclized.   

\subsection{Cycle space of a graph} 
Let $G$ be a graph. The chain group  $C_1=C_{1}(G)=\mathbb{Z}^{|E(G)|}$ is generated by the oriented edges $\{[e]\mid e \in E(G)\}$, and $C_0=C_{0}(G)=\mathbb{Z}^{|V(G)|}$ by the vertex set $V(G)$. 
An element $x\in C_{1}$ may be represented either as a column vector $x=(n_{e})_{e\in E(G)}$ or as a \emph{formal} sum $x=\sum_{e\in E(G)}n_{e}[e]$ with $n_{e}\in \mathbb{Z}$ for all $e\in E(G)$. The elements of $C_{0}$ will be represented similarly. 
The map $\partial_{1}=\partial_{1}(G):C_{1}\rightarrow C_{0}$ is the \emph{incidence matrix} of $G$ 
which is an integer matrix with the rows and columns indexed by $V(G)$ and $E(G)$, respectively.
Note that any column of $\partial_{1}$ indexed by a loop is a zero column. 

The \emph{cycle space} of $G$ is $\ker\partial_{1}=Z_{1}=H_1(G)$. If $C$ is a cycle as a subgraph of $G$, then we have $z=\sum_{e\in E(C)}\epsilon_{e} [e]\in Z_{1}$ by suitably choosing the coefficients $\epsilon_{e}=\pm 1$. It is well-known that the rank of the cycle space for a connected $G$ equals the \emph{corank} $|E(G)|-|V(G)|+1$ of $G$ (refer to \cite{Bi}).  In this case, an important basis for $Z_{1}$ is given as follows. Let $T$ be a spanning tree in $G$.  For each $e\in E(G)\setminus E(T)$, there is a unique cycle in $T\cup e$ which must contain $e$.  Let $z_{e}$ denote the element in $Z_{1}$ that corresponds to this cycle (with $\epsilon_{e}=1$). Then the collection $\{z_{e}\mid e\in E(G)\setminus E(T)\}$ is a basis for $Z_{1}$ (refer to \cite{Bi}).  Hence, every $z\in Z_{1}$ is written uniquely as 
\begin{equation}\label{basis}
z=\sum_{e\in E(G)\setminus E(T)}m_{e}z_{e}
\end{equation} 
where $m_{e}$ is the coefficient of $[e]$ in $z$ for $e\in E(G)\setminus E(T)$.   

\subsection{Winding number}

Let $G$ be a connected graph and $\partial_{1}$ its incidence matrix. A \emph{unicyclizer} of $G$ is an integer matrix $\partial$ such that 
\begin{itemize}
\item[(1)] $\ker\partial=0$
\item[(2)] $\partial_{1}\partial=0$
\item[(3)] $\rk(\ker\partial_{1}/\im\partial)=1$ . 
\end{itemize}
 
The pair $\mathscr{A}=(G,\partial)$ will be called a \emph{unicyclization} of $G$ to which we associate a chain complex $\{A_{i}, \partial_{i}\}$ where 
$A_{i}=C_{i}(G)$ for $i=0,1$ with $\partial_{1}=\partial_{1}(G)$, $A_{2}=\mathbb{Z}^{r}$ with $r=\rk\partial$, and $\partial_{2}=\partial$. Also, we have $\rk H_{1}(\mathscr{A})=1$. Further, we will adopt the following notations for $\mathscr{A}$:
$$\mathcal{T}(\mathscr{A})= \mathcal{T}(G),\quad \mathcal{U}(\mathscr{A})= \mathcal{U}(G),\quad \mbox{and} \quad k(\mathscr{A})= k(G)\,.$$

%\begin{mydef}
%The pair $\mathscr{A}=(G, \partial_2)$ is called an Aug-complex, where $G$ is a connected (directed) graph (thus, we may fix $\partial_1:A_1(G)=\mathbb{Z}^{m_1} \rightarrow A_0(G)=\mathbb{Z}^{m_0}$), an integer matrix $\partial_2:A_2\equiv\mathbb{Z}^{m_2} \rightarrow A_1(G)=\mathbb{Z}^{m_1}$ satisfies $\partial_1 \partial_2=0$.
%\end{mydef}
%\begin{mydef}
%Let $\mathscr{A}=(G, \partial_2)$ be an Aug-complex. Define $H(\mathscr{A})=\dfrac{\ker \partial_1 }{\im \partial_2}$ as the homology of $\mathscr{A}$. An Aug-complex $\mathscr{A}$ is called rank 1 if $\rk H(\mathscr{A})=1$.
%\end{mydef}
%
%First, we do not like to have ambient information of $A_1(G)$ to represent the elements of $\ker\partial_1 \leq A_1(G)$. In other words, we will use a basis to reduce ambient rows in $\partial_2$.
%
%\begin{mydef}
%The $\beta$ is called a basis of an Aug-complex $\mathscr{A}=(G, \partial_2)$ if $\beta$ is a $\mathbb{Z}$-basis of $\ker\partial_1$.
%\end{mydef}
\begin{ex}
For a CW-complex $X$ with $\rk H_1(X)=1$, there is a unicyclization $(X^1,\overline{\partial_2})$ where $\overline{\partial_2}$ a full-rank submatrix of $\partial_2$.
\end{ex}
Let $\beta$ be a basis of the cycle space $\ker\partial_{1}=Z_{1}=\mathbb{Z}^{m}$ ($m>0$). A cycle $z\in Z_1$ will be denoted $[z]_{\beta}\in\mathbb{Z}^{m}$ when it is written with respect to $\beta$. Similarly, $[\partial]_{\beta}$ means the  matrix whose columns are those of $\partial$ written with respect to $\beta$. From conditions (1) and (3) for $\partial$, it is clear that $[\partial]_{\beta}$ is $m\times (m-1)$ with linearly independent columns. 
%\begin{mydef}
%A cycle $C \in \ker\partial_1$ can be written $[C]_{\beta} \in \mathbb{Z}^m$ with respect to a basis $\beta$ where $m=\rk\ker\partial_1$.  %due to $\im\partial_2\leq\ker\partial_1$.
%\end{mydef}
%
%Second, we don't want ambient columns in the matrix form of $\partial_2$ with respect to the elementary column operations over $\mathbb{Z}$. Thus, we suggest the following definition. We can find a finite matrix $\overline{\partial_2}$ whose columns are vectors in $A_1(G)$ such that $\im\overline{\partial_2}=\im\partial_2$ and $\rk \overline{\partial_2}=$ (the number of columns of $\overline{\partial_2}$). We hope to use $\overline{\partial_2}$ instead of $\partial_2$. Note that $\overline{\partial_2}$ is unique upto the elementary column operations over $\mathbb{Z}$.
%
%\begin{mydef}
%An Aug-complex $\mathscr{A}=(G, \partial_2)$ is called reduced if $\rk \partial_2$ = [the number of the columns of $\partial_2$].
%\end{mydef}
%
%\begin{prop}\label{mat1}
%Let $\mathscr{A}=(G, \partial_2)$ be a reduced rank 1 Aug-complex with a basis $\beta$. Then $\partial_2$ can be represented as the matrix $[\partial_2]_{\beta}$ whose size is $m \times (m-1)$ where $m$ is the rank of $\ker \partial_1$.
%\end{prop}
%Due to the previous proposition \ref{mat1}, we can use the determinant to measure the winding number of a cycle $C\in\ker\partial_1$ in a homology.
We are ready to present the main definition of this section.
\begin{mydef} (Winding number of a cycle) Let $\mathscr{A}=(G, \partial)$ be a unicyclization of a connected graph $G$, and $\beta$ a basis of the cycle space $Z_{1}(G)$. We define the winding number $w_{\mathscr{A}}: Z_1(G) \rightarrow \mathbb{Z}$ relative to $\beta$ to be  
$$w(z)=w_{\mathscr{A}}(z)=\det([z]_{\beta},[\partial]_{\beta})\,.$$ 
\end{mydef}
%Note that this definition is independent of the choice of $\beta$.  
When a basis $\beta$ of $Z_{1}(G)$ is fixed, we may omit $\beta$ from the notation and write $w(z)=\det(z,\partial)$. Also, note that $w(z)=0$ for $z\in \im\partial$, and thus $w_{\mathscr{A}}(\cdot)=\det(\cdot, \partial): H_{1}(\mathscr{A}) \rightarrow \mathbb{Z}$ is well-defined. 
%(The meaning of winding number and the definition of torsion-free winding number map will be given later in Section (Winding number and torsion-free version)).
\subsubsection*{What winding number measures}
We can get the Smith normal form of $[\partial]_\beta$ as follows.
\begin{equation}
[\partial]_\beta=S\left( \begin{array}{c}
D  \\
0\cdots 0 
\end{array}
\right)T
\end{equation}
where $m=\rk Z_1(G)$, $S\in GL(m;\mathbb{Z})$, $T\in GL({m-1};\mathbb{Z})$, and $D$ is the diagonal matrix with diagonal entries $d_1,\cdots, d_{m-1}$ with conditions $d_i\mid d_{i+1}$ and $d_i>0$. Therefore, we get
\begin{equation}
H_1(\mathscr{A})\cong \mathbb{Z}_{d_1}\oplus \cdots \oplus \mathbb{Z}_{d_{m-1}} \oplus \mathbb{Z}
\end{equation}
Denote the size of its torsion part by $\tau=d_1 \times \cdots \times d_{m-1}$.

Let's expand $T$ to $\overline{T}\in GL(m;\mathbb{Z})$ by inserting 1 at $(m,m)$ entry and 0 otherwise. Therefore, we get ${|w_{\mathscr{A}}(C)|}={|\det([C]_\beta,[\partial]_\beta)|}={|\det([\partial]_\beta,[C]_\beta)|}={| \det(S^{-1})\cdot \det([\partial]_\beta,[C]_\beta)\cdot \det(\overline{T}^{-1})|}={| \det\left( \begin{array}{cc}
\begin{array}{c}
D \\
0 \cdots 0
\end{array} & S^{-1}[C]_\beta
\end{array}
\right)|}=\tau \cdot |$[the last entry of the vector $S^{-1}[C]_\beta$]$|$ because ${|\det(S)|}={|\det(T)|}=1$.
Now, we can see $w_{\mathscr{A}}(C)$ measure how many the given cycle $C$ turns around the $\mathbb{Z}$ part in homology. And we get the following proposition.
\begin{prop}
Let $\mathscr{A}=(G, \partial)$ be a unicyclization. We have a winding number map $w_{\mathscr{A}}:\ker\partial_1 \rightarrow\mathbb{Z}$. Then, the image of $w_{\mathscr{A}}$ is $\tau \mathbb{Z}$.
\end{prop}
\begin{proof}
Under the previous argument, $S^{-1}[C]_\beta$ can be any vector because $S$ is a basis transformation.
\end{proof}

\subsection{Standard harmonic cycle}
Let us introduce the main object of study in this paper. Recall that for a connected graph $G$, every $Y\in \mathcal{U}(G)$ contains a unique cycle $C_{Y}$, and  we will let $z_{Y}$ denote the correponding element in $Z_{1}(G)$.
\begin{mydef}(Standard harmonic cycle) Let $\mathscr{A}=(G, \partial)$ be a unicyclization of a connected graph $G$. The standard harmonic cycle of $\mathscr{A}$ is
\begin{equation}
\lambda_{\mathscr{A}} \equiv \sum_{Y\in \mathcal{U}(G)} w_{\mathscr{A}}(z_Y) \cdot z_Y
\end{equation}
as an element of $Z_{1}(G)=H_{1}(G)$.
\end{mydef}
Note that $\lambda_{\mathscr{A}} $ is independent of the orientations of $C_{Y}$ 
for all $Y\in \mathcal{U}(G)$. Indeed, for each $Y\in \mathcal{U}(G)$, the \emph{weighted} cycle $w(z_Y)\cdot z_Y$ is independent of the orientation of $C_{Y}$ because $w(z_Y)\cdot z_Y=\det(z_Y, \partial)\cdot z_Y=\det(-z_{Y}, \partial)\cdot (-z_{Y})=w(-z_Y)\cdot(-z_Y)$. 
%Later, we shall show that $\lambda_{\mathscr{A}}$ is indeed a harmonic cycle in that it is both a cycle and a cocycle.
\begin{prop}
Let $\mathcal{C}(G)$ be the set of all cycles in $G$. Under the previous notation, we have
\begin{equation}
\lambda_{\mathscr{A}}=\sum_{C\in \mathcal{C}(G)} k(G/C)\cdot w_{\mathscr{A}}([C]) \cdot [C]
\end{equation}
where the sum is over all cycles $C\in \mathcal{C}(G)$ together with $[C]\in Z_{1}(G)$, and $G/C$ is obtained by deleting all edges in $C$ and identifying all vertices in $C$ to a vertex.
\end{prop}
\begin{proof} Given $C\in\mathcal{C}(G)$, let $\mathcal{U}(G)_{C}=\{Y\in \mathcal{U}(G)\mid C_{Y}=C\}$.  The map $Y\mapsto Y/C$ defines a bijection from $\mathcal{U}(G)_{C}$ to $\mathcal{T}(G/C)$. Since $\mathcal{U}(G)=\amalg_{C\in\mathcal{C}(G)}\mathcal{U}(G)_{C}$, the proposition follows from the definition of $\lambda_{\mathscr{A}}$. Details will be omitted. 
\end{proof}
\begin{ex}
Given a connected cell complex $X$ with $\rk H_1(X)=1$, let $\partial$ be a matrix whose columns form a basis for $\im\partial_{2}$ for $X$. Then $\mathscr{A}=(X^{(1)}, \partial)$ is a unicyclization of the 1-skeleton $X^{(1)}$ of $X$ such that $H_1(X)=H_1(\mathscr{A})$.  Therefore, we can use and $w_{\mathscr{A}}$ and $\lambda_{\mathscr{A}} $.
\end{ex}

Later, we will show in Section (Main theorem A and B) that $\lambda_{\mathscr{A}}$ is a non-zero harmonic cycle in that it is both a cycle and a cocycle.

\begin{ex}
The leftmost picture is the standard harmonic cycle. The middle picture shows some cycletrees. The third figure shows their unique cycles (in red), oriented in the same direction. The sum of the red cycles on each edge is the standard harmonic cycle.

\begin{center}
\includegraphics[scale=0.25]{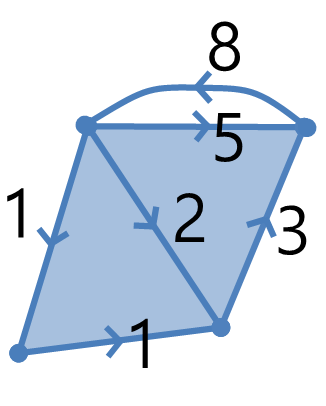},\ \ \ \
\includegraphics[scale=0.2]{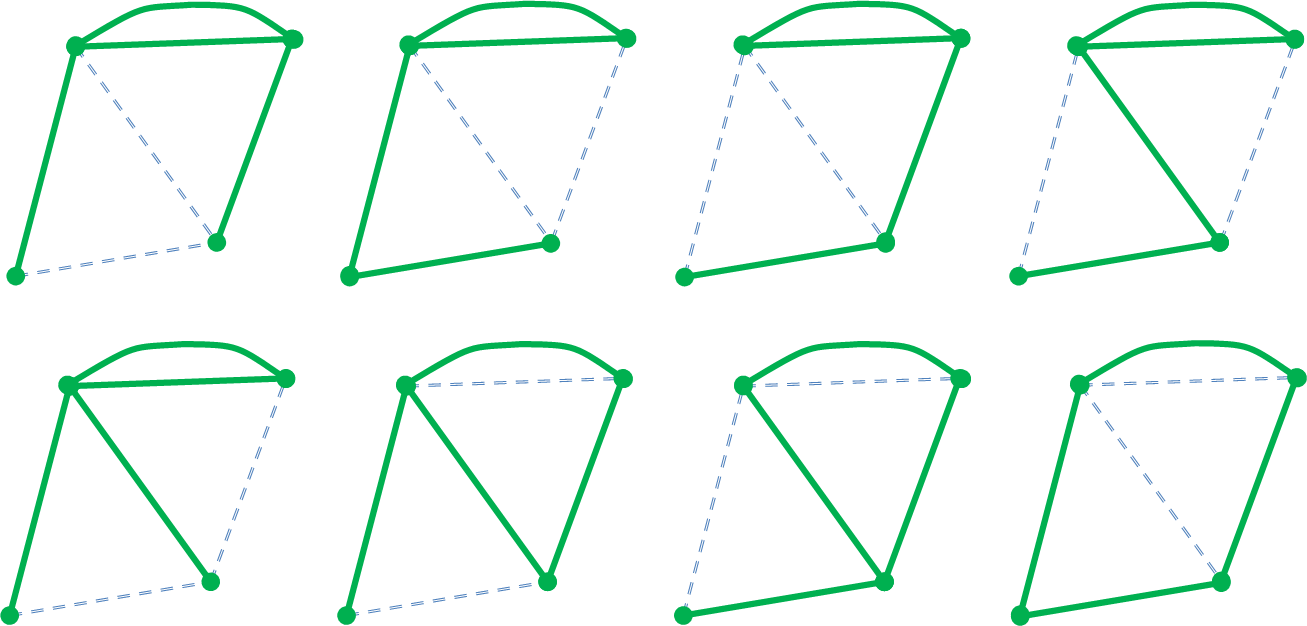}, \ \ \ \ \ 
\includegraphics[scale=0.2]{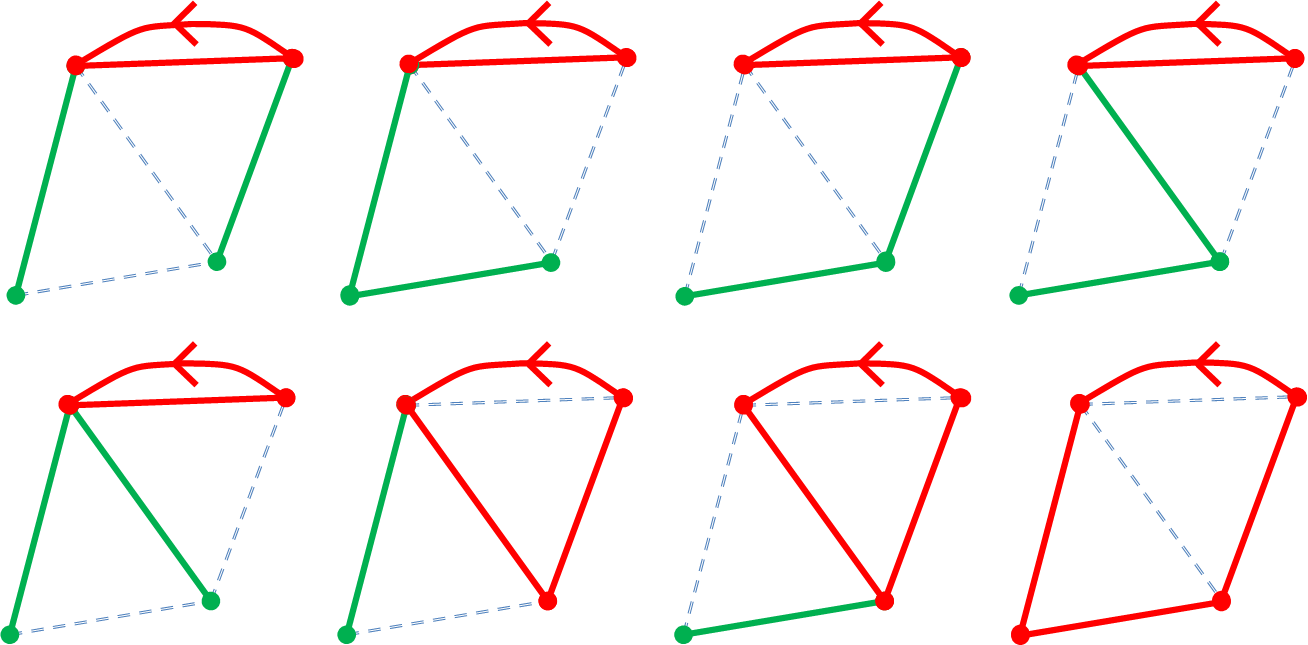}
\captionof{figure}{An example of standard harmonic cycle}
\end{center}

\begin{remark}
An example of a cycletree with zero winding number is given below. 
%at its the circuit parts like a cycletree in figure 2 for the convenience.
\end{remark}

\begin{center}
\includegraphics[scale=0.2]{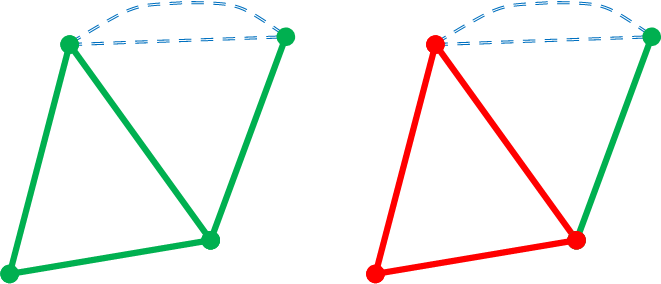}
\captionof{figure}{ The red circuit has zero winding number }
\end{center}
\end{ex}

\begin{ex}
We can collect cycletrees with the same unique cycle as in the below figure. $\#CT$ means the number of the cycletree with a given cycle. %This notation makes us not need to write every cycletree.
\begin{center}\includegraphics[scale=0.2]{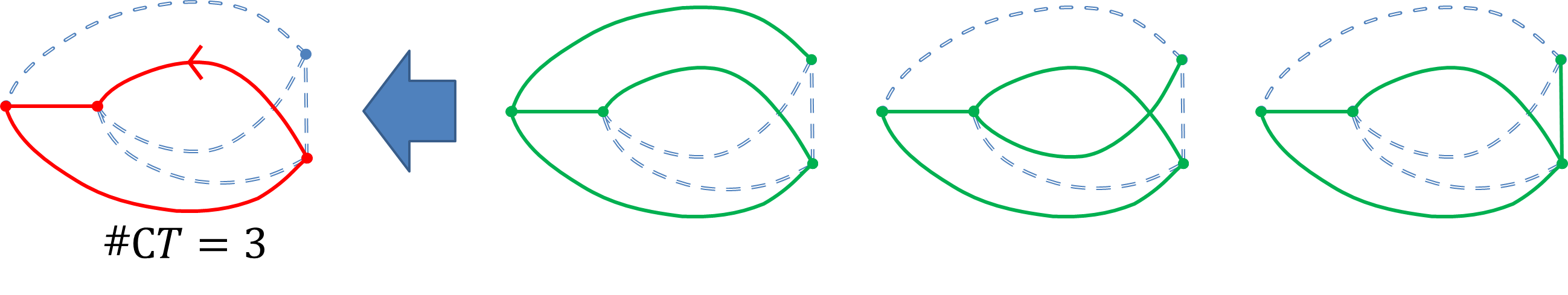} \\
\captionof{figure}{There are three cycletrees having the same cycle.}
\end{center}
Here is a CW-complex which shows that the winding number is an essential ingredient. $\#W$ stands for the winding number of a given cycle. Therefore, we get the following equation.
\begin{center}
\includegraphics[scale=0.15]{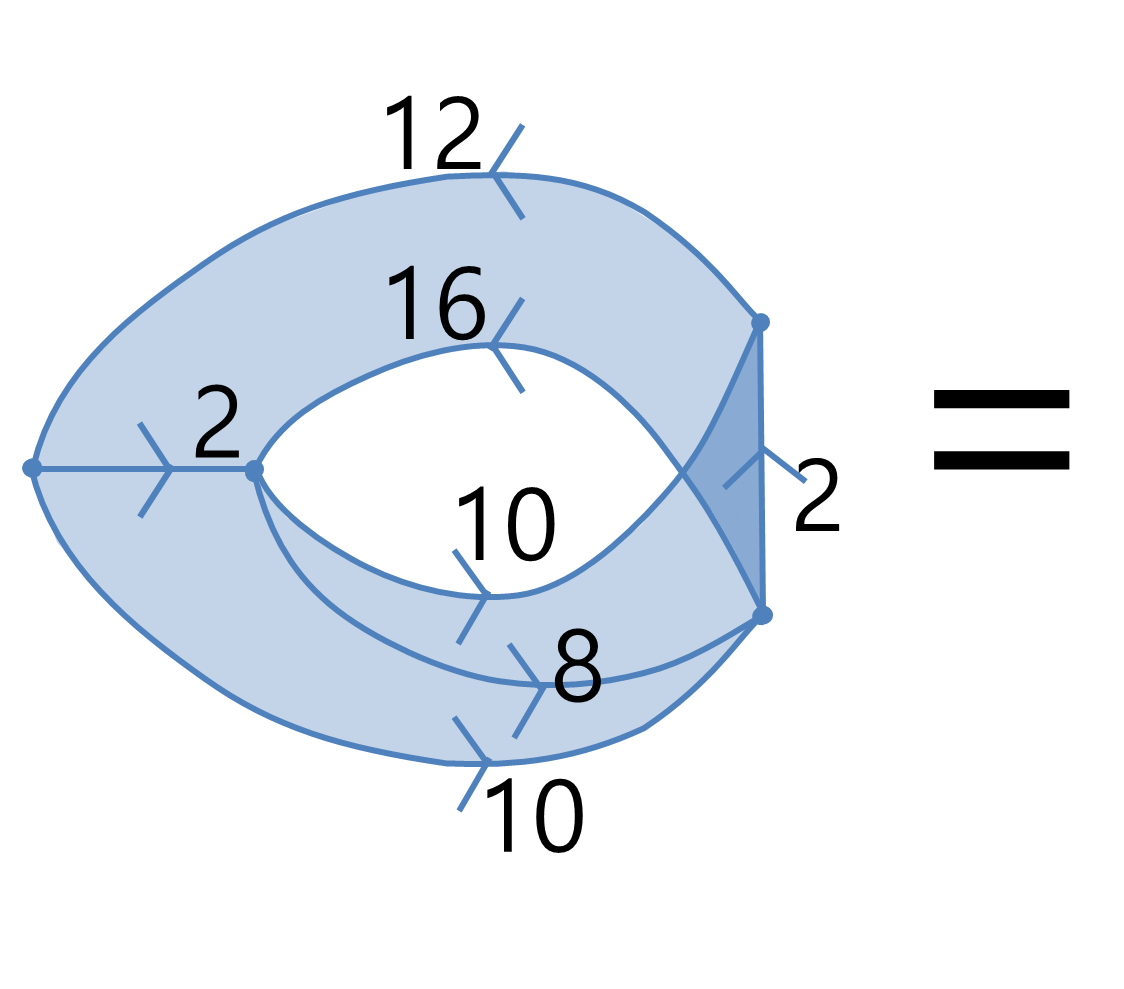}
\includegraphics[scale=0.14]{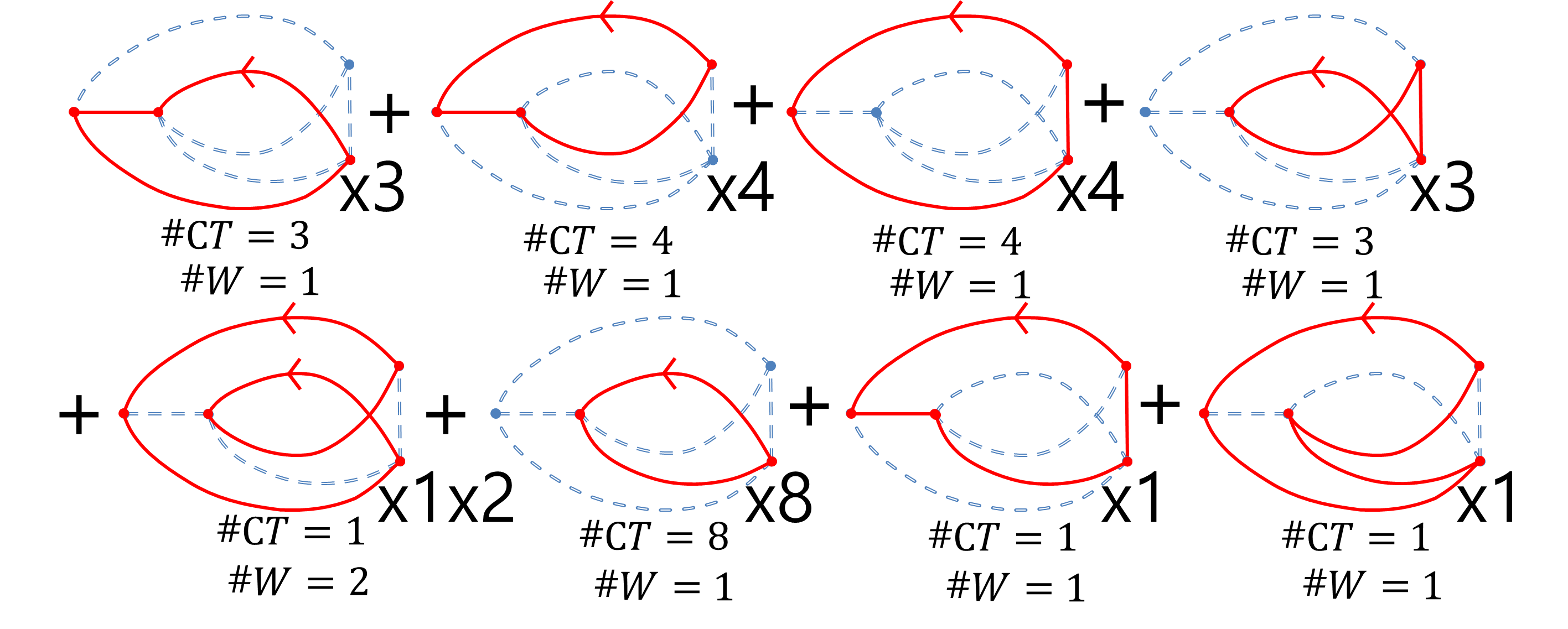}
\captionof{figure}{The standard harmonic cycle}
\end{center}
\end{ex}

\subsection{Relations between unicyclizations}
%In this subsection, we will discuss relations among .

%First, with an Aug-complex $\mathscr{A}=(G, \partial_2)$ with a basis $\beta$, we will show that the different choice of $\partial_2$, and $\beta$ doesn't change the original Aug-complex that much.

%\begin{notation}Write briefly an Aug-complex $\mathscr{A}=(G, \partial_2)$ having a fixed basis $\beta$ as $(G, \partial_2, \beta)$.\end{notation}

\begin{prop}\label{mat2}
Let $\mathscr{A}=(G,\partial)$ and $\mathscr{A}'=(G, \partial')$ be unicyclizations of a connected graph $G$ such that $\im \partial=\im\partial'$. Let $\beta$ and $\beta'$ be bases of $Z_{1}(G)$. Then $w_{\mathscr{A}}(\cdot)=\det(\cdot, [\partial]_{\beta})$ relative to $\beta$ and $w_{\mathscr{A}'}(\cdot)=\det(\cdot, [\partial']_{\beta'})$ relative to $\beta'$ are the same up to sign. Of course, $H_1(\mathscr{A})=H_1(\mathscr{A}')$.
\end{prop}
\begin{proof}
If $z\in \im\partial=\im\partial'$, then $w_{\mathscr{A}}(z)=w_{\mathscr{A}'}(z)=0$. So suppose $z\notin \im\partial.$ Clearly, the matrices $M=([z]_{\beta'},[\partial]_{\beta'})$ and $M'=([z]_{\beta'},[\partial']_{\beta'})$ have the same column space, and the set $\alpha$ of columns of $M$ forms a basis of that space and the same is true for the set $\alpha'$ of columns of $M'$.  Hence
$$\det([z]_{\beta},[\partial]_{\beta})=\det B\cdot\det([z]_{\beta'},[\partial]_{\beta'})=\det B\cdot\det([z]_{\beta'},[\partial']_{\beta'})\cdot\det A$$
where $B$ is a change of basis matrix for $\beta$ and $\beta'$ and $A$ is a change of basis matrix for $\alpha$ and $\alpha'$. Since both $B$ and $A$ are invertible integer matrices, each has determinant $\pm1$. 
\end{proof}   

\begin{cor}
Under the same assumptions of Proposition \ref{mat2}, we have $w_{\mathscr{A}}(\cdot)=\pm w_{\mathscr{A'}}(\cdot)$ and $\lambda_{\mathscr{A}}=\pm \lambda_{\mathscr{A'}}$.
\end{cor}
\begin{proof} In the proof of the Proposition\ref{mat2}, the matrix $B$ is fixed, hence so is $\det B$. Also, it is clear that $A$ is the same for any choice of $z\in Z_{1}(G)$  for the first columns of $M$ and $M'$. Therefore $\det A$ is also fixed, and the results follow.
\end{proof}

%To choose a basis consistently without sign problem while changing $G$, we will define a basis of an Aug-complex $\mathscr{A}$ through the following proposition.

%\begin{prop} Let $\mathscr{A}=(G, \partial_2)$ be an Aug-complex. For a spanning tree $T$ of $G$, we define $\beta_T$ as an ordered set of oriented fundamental cycles (or circuits) with respect to the spanning tree $T$. Then, $\beta_T$ is a basis of $\ker\partial_1$. In other words, $\ker \partial_1= \mathbb{Z}c_1 \times \cdots \times \mathbb{Z}c_m( \cong \mathbb{Z}^m)$ where $\beta_T = \{c_1,\cdots,c_m \}$. \end{prop}

%\begin{proof}$T$ is a spanning tree in the graph $G$. And $c_1,\cdots,c_m$ in $O$ are all fundamental cycles with respect to the spanning tree $T$.\end{proof}

%The orientation on the edges of $G$ can be given arbitrarily, and it effects on a basis of $A_1(G)$ as the change of sign.

%The order of the set $\beta_T$ effects on the order of $\mathbb{Z}$ in $\mathbb{Z}^m=\mathbb{Z}\times \cdots \times \mathbb{Z}$, and the orientation of fundamental cycles does on the sign of $\mathbb{Z}$.

%\begin{ex} With a CW-complex $X$, we can construct an Aug-complex $\mathscr{A}=(X^1,\partial_2,\beta_T)$ by choosing a spanning tree $T$ and the order and orientation of its fundamental cycles. \end{ex}

The following propositions reveal the relation between two unicyclizations when one is constructed from the other by edge contraction or deletion.

For a convenience, shortly write a unicyclization $\mathscr{A}=(G, \partial)$ with a certain choice of a basis $\beta$ for $Z_1(G)$ as $\mathscr{A}=(G, \partial)$ with $\beta$. Moreover, let $G_c=G/\sigma$ and $G_d=G\setminus \sigma$ with the corresponding incidence matrices $\partial_{1,c}$ and $\partial_{1,d}$ for a given edge $\sigma$.

\begin{prop} (Contraction)
Let $\mathscr{A}=(G, \partial)$ with $\beta$ be a unicyclization and $\sigma$ be an edge in $G$. Assume $\sigma$ is not a loop. Then, there is a unicyclization $\mathscr{A}_c=(G_c, \partial_c)$ with $\beta_c$ such that $[C]_{\beta}=[C_c]_{\beta_c}$ and $w_{\mathscr{A}}(C)=w_{\mathscr{A}_c}(C_c)$ where $C_c\in A_1(G_c)$ is a vector deleted $\sigma$ entry from a vector $C\in \ker\partial_1$.
\end{prop}
\begin{proof}
From a natural map $G\rightarrow G_c=G/\sigma$, we get the isomorphism $\ker \partial_1 \cong \ker \partial_{1,c}$ because $\sigma$ is not a loop. By this isomorphism, we can find the corresponding basis $\beta_c$ of $\ker\partial_{1,c}$ to the basis $\beta$ of $\ker\partial_1$. Define $\partial_c$ by $[\partial_c]_{\beta_c}=[\partial]_\beta$. We can easily check unicyclization conditions for $\mathscr{A}_c$ due to the previous isomorphism. Moreover, we can show $[C]_{\beta}=[C_c]_{\beta_c}$ rigorously by using the basis \ref{basis} after changing $\beta$ with Proposition \ref{mat2}. Finally, $w_{\mathscr{A}}(C)=\det([C]_{\beta}, [\partial]_\beta)=\det([C_c]_{\beta_c}, [\partial_c]_{\beta_c})=w_{\mathscr{A}_c}(C_c)$.
\end{proof}

%\begin{remark} For a given $C \in \ker\partial_1$, define $C_c \in \ker\partial_{1,c}$ by $[C]_{\beta}=[C_c]_{\beta_c}$.\end{remark}

Let's introduce some notations. $\partial|_{\sigma}$ means the row vector of a matrix $\partial$ at the edge $\sigma$. Let $n_\sigma$ be the gcd of the entries of a vector $\partial|_{\sigma}$. Note that $\sigma$ is not a bridge if $\partial|_{\sigma} \neq \overrightarrow{0}$. Moreover, when $\partial|_{\sigma} = \overrightarrow{0}$, $\rk H_1(\mathscr{A})$ becomes 0 if $\sigma$ is removed.

\begin{prop} (Deletion)
Let $\mathscr{A}=(G, \partial)$ with $\beta$ be a unicyclization and $\sigma$ be an edge in $G$. Assume $\partial|_{\sigma} \neq \overrightarrow{0}$. Then, there is a unicyclization $\mathscr{A}_d=(G_d, \partial_d)$ with $\beta_d$ such that $[C]_{\beta}=[C_d]_{\beta_d}$ and $w_{\mathscr{A}}(C)=n_{\sigma}\cdot w_{\mathscr{A}_d}(C_d)$ where $C_d\in A_1(G_d)$ is a vector deleted $\sigma$ entry from a vector $C\in \ker\partial_1$ when $C$ has a zero coefficient at $\sigma$.
\end{prop}
\begin{proof}
We can find a spanning tree $T$ such that $\sigma$ is not in $T$. Without loss of generality, we can assume $\beta$ is the set of the fundamental cycles with respect to $T$ by Proposition \ref{mat2}. Moreover, assume that the fundamental cycle corresponds to $\sigma$ be the last one in the order of $\beta$.

From a natural map $G_d=G\setminus \sigma \rightarrow G$, we have the isomorphism $\ker \partial_1 \cong \ker \partial_{1,d}\times \mathbb{Z}$  where $\mathbb{Z}$ stands for the lost fundamental cycle when $G$ becomes $G_d$. From this map, we get $\beta_d$ from $\beta$. By the column elementary operations over $\mathbb{Z}$ which does not change the sign, we can transform $[\partial]_{\beta}$ to $[\widehat{\partial}]_{\beta}$ such that
\begin{equation}
[\widehat{\partial}]_{\beta}=\left( \begin{array}{cc}
[\partial_{d}]_{\beta_{d}} & * \\
0\cdots 0  & n_\sigma
\end{array}
\right).
\end{equation}
Define $\partial_{d}$ by the previous equation. Conditions for $\mathscr{A}$ to be a unicyclization can be easily checked. Moreover, $[C]_{\beta}=[C_d]_{\beta_d}$ can be easily checked due to the basis \ref{basis}. Finally,
\begin{equation}
\det([C]_{\beta},[\partial]_{\beta})=\det([C]_{\beta},[\widehat{\partial}]_{\beta})=n_\sigma\cdot \det([C_d]_{\beta_d},[\partial_{d}]_{\beta_d})
\end{equation}
where the entry of $C\in \ker \partial_1$ at $\sigma$ is zero.
\end{proof}

\begin{mydef}
Let's call the previous gcd $n_\sigma$ as the winding difference. 
\end{mydef}

\subsection{Relations between standard harmonic classes}
In this subsection, we will observe how the sub-parts of the standard harmonic cycle are related with respect to the inner product when a unicyclization has an edge contraction and deletion. This observation will be used in the proof of the main theorem.
\\

Now, we will divide the standard harmonic cycle into two parts with respect to an edge $\sigma$ in a unicyclization $\mathscr{A}$. 
\begin{equation}
\lambda_{\mathscr{A}} \equiv \lambda (\mathscr{A}) = \lambda_\sigma (\mathscr{A})+\lambda_{-\sigma}(\mathscr{A})
\end{equation}
where
\begin{equation}
\lambda_\sigma (\mathscr{A}) \equiv \sum_{Y\in \mathcal{U}(G), \sigma \in Y} w_{\mathscr{A}}(z_Y) \cdot z_Y,
\end{equation}
\begin{equation}
\lambda_{-\sigma} (\mathscr{A}) \equiv \sum_{Y\in \mathcal{U}(G), \sigma \notin Y} w_{\mathscr{A}}(z_Y) \cdot z_Y.
\end{equation}

\begin{prop}
Let $\mathscr{A}=(G, \partial)$ with $\beta$ be a unicyclization and $\sigma$ be an edge of $G$.
\begin{enumerate}
\item If $\sigma$ is a loop, then $\lambda_{\sigma}(\mathscr{A})=m\cdot [\sigma]\in A_1(G)$ for some $m\in \mathbb{Z}$.
\item If $\sigma$ is not a loop, then $[\lambda_{\sigma} (\mathscr{A})]_{\beta}=[\lambda(\mathscr{A}_c)]_{\beta_c}$.
\item If $\partial|_{\sigma}=\overrightarrow{0}$, then $\lambda_{-\sigma}(\mathscr{A})=0$.
\item If $\partial|_{\sigma} \neq \overrightarrow{0}$, then $[\lambda_{-\sigma} (\mathscr{A})]_{\beta}=n_\sigma \cdot [\lambda(\mathscr{A}_d)]_{\beta_d}$, where $n_\sigma$ is the winding difference.
\end{enumerate}
\end{prop}

\begin{proof}
We can check Statement 1 and 3 from their definitions. Statement 2 and 4 hold because $w_{\mathscr{A}} \simeq w_{\mathscr{A}_c}$ and $w_{\mathscr{A}} \simeq n_\sigma \cdot w_{\mathscr{A}_d}$. %To be specific, $w_{\mathscr{A}}(C)=w_{\mathscr{A}_c}(C_c)$ where $C\in \beta_T$ and $C_c$ is the corresponding circuit in $B_{T,c}$. $w_{\mathscr{A}}(C)=w_{\mathscr{A}_d}(C_d)$ where $C\in \beta_T$ and $C_d\in B_{T,d}$. 
\end{proof}

\begin{cor}\label{cor1}
Let $\mathscr{A}=(G, \partial)$ with $\beta_T$ be a unicyclization, $C$ be a cycle in $G$ and $\sigma$ be an edge in $G\setminus C$.
\begin{enumerate}
\item If $\sigma$ is a loop, then $C\circ \lambda_{\sigma} (\mathscr{A})=0$,
\item If $\sigma$ is not a loop, $C\circ \lambda_{\sigma} (\mathscr{A})=C_c\circ \lambda(\mathscr{A}_c)$,
\item If $\partial|_{\sigma} = \overrightarrow{0}$, then $C\circ \lambda_{-\sigma} (\mathscr{A})=0$,
\item If $\partial|_{\sigma} \neq \overrightarrow{0}$, then $C\circ \lambda_{-\sigma} (\mathscr{A})=n_\sigma \cdot C_d\circ \lambda(\mathscr{A}_d)$,

\end{enumerate}
where $C_c$ is a cycle (or the sum of two cycles) in $\mathscr{A}_c$ and $C_d$ is a cycle in $\mathscr{A}_d$ and $n_\sigma$ is the winding difference.
\end{cor}

\section{Main theorem A and B}

Now, we will state and prove a unicyclization version of the harmonic cycle decomposition theorem.

%\subsection{Statement of main theorem A and B}
%Let $X$ be a connected finite CW-complex with $\rk H_1(X)=1$. As mentioned in Section 4, we have a unicyclization $\mathscr{A}=(X^1,\overline{\partial_2})$. And we have defined the winding number map $w_{\mathscr{A}}$ associated to a basis $\beta$ and the standard harmonic cycle $\lambda_{\mathscr{A}}$. Define the winding number map of $X$ as $w_X= w_{\mathscr{A}}$ and the standard harmonic cycle of $X$ as $\lambda_X\equiv\lambda(X)=\lambda_{\mathscr{A}}$. Note that $w_X$ and $\lambda(X)$ are unique upto sign.

%\begin{thm}(Theorem A, winding number with harmonic cycle)\\
%Let $X$ be a connected finite CW-complex with $\rk H_1(X)=1$. For a cycle $C$ in $Z_1(X)\leq A_1(X)$ (or, $Z_1(X;\mathbb{R})\leq A_1(X;\mathbb{R})$), then 
%\begin{equation}
%C\circ \lambda_X = w_X(C)\cdot k_X
%\end{equation}
%where $\circ$ is the inner product in $A_1(X)$ (or, $A_1(X;\mathbb{R})$), and %$k_X\equiv k(X)$ is the number of spanning trees.
%\end{thm}

%\begin{thm}
%(Theorem B, the CW-harmonic cycle decomposition)\\
%Let X be a connected finite CW-complex with $\rk H_1(X)=1$. Then $\lambda_X$ %is actually a non-zero element of $H^h_1(X)\subsetneq H^h_1(X;\mathbb{R})$. In other words, every elements of $H^h_1(X;\mathbb{R})$ are the multiple of 
%\begin{equation}
%\lambda_X \equiv \sum_{Y \in \mathcal{U}(X^1)} w_{X}(z_Y)\cdot z_Y.
%\end{equation}

%\end{thm}

\subsection{Theorem A}
%It is enough to prove Theorem A for a unicyclization version. 
\begin{thm} (Theorem A, inner product formula)
Let $\mathscr{A}=(G,\partial)$ be a unicyclization. Define 
\begin{equation}
\lambda_{\mathscr{A}}=\lambda(\mathscr{A})= \sum_{Y \in \mathcal{U}(\mathscr{A})}w_{\mathscr{A}}(z_Y)\cdot z_Y
\end{equation}
For $C\in ker\partial_1$, then 
\begin{equation}
C\circ \lambda_{\mathscr{A}} = w_{\mathscr{A}}(C)\cdot k_{\mathscr{A}}
\end{equation}
where $\circ$ is the inner product and $w_{\mathscr{A}}(C)=\det(C,\partial)$.
\end{thm}

\begin{proof}
We can assume that $C$ is a closed path rather than a general cycle (i.e., $C\in ker\partial_1$) due to the linearity of the inner product $(-) \circ\lambda_{\mathscr{A}}$ and the determinant $\det(-,\partial)$.

\
We will use an induction proof on the number of edges to prove this theorem. As the induction step, use the following lemmas for each four cases to show that $C\circ \lambda_{\mathscr{A}}= w_{\mathscr{A}}(C)\cdot k(G)$, under the motivation that we will get rid of an edge $\sigma \in E(G)\setminus E(C)$. Then, it is sufficient to prove the last proposition \ref{mat3}, the induction base case when $C$ has every edge in $G$. Note that $\lambda(\mathscr{A})=\lambda_{\sigma}(\mathscr{A})+\lambda_{-\sigma}(\mathscr{A})$.
\end{proof}

The following four lemmas and one proposition are under the assumptions of the preceding main proof. 
\begin{lem}
Assume $\sigma$ is a loop, and $(\partial)|_{\sigma} \neq  \overrightarrow{0}$. Then, $C\circ \lambda_{\mathscr{A}}=n_\sigma \cdot C_{d}\circ \lambda_{\mathscr{A}_d}=n_\sigma \cdot w_{\mathscr{A}_d}(C_d) \cdot k_{\mathscr{A}_d}= w_{\mathscr{A}}(C)\cdot k_{\mathscr{A}}$.
\end{lem}

\begin{proof}
Due to the assumptions about $\sigma$, from $\mathscr{A}=(G,\partial)$ with $\beta$, we can construct $\mathscr{A}_d=(G_d,\partial_d)$ with $\beta_d$ and the winding difference $n_\sigma$. We showed that $C\circ \lambda_{\mathscr{A}}=n_\sigma \cdot C_d\circ \lambda_{\mathscr{A}_d}$ at Corollary \ref{cor1}. From the induction hypothesis, $C_d\circ \lambda_{\mathscr{A}_d}=w_{\mathscr{A}_d}(C_d)\cdot k_{\mathscr{A}_d}$ holds. We know $n_\sigma \cdot w_{\mathscr{A}_d}(C_d)=w_{\mathscr{A}}(C)$ and $k_{\mathscr{A}_d}=k_{\mathscr{A}}$.
\end{proof}

\begin{lem}
Assume $\sigma$ is a loop, and $(\partial)|_{\sigma} = \overrightarrow{0} $. Then, $C\circ \lambda_{\mathscr{A}}=0=w_{\mathscr{A}}(C) \cdot k_{\mathscr{A}}$.
\end{lem}
\begin{proof}
Since $C\circ \lambda_{\sigma}(\mathscr{A})=0$ and $C\circ \lambda_{-\sigma}(\mathscr{A})=0$ by Corollary \ref{cor1}, we get $C\circ\lambda_{\mathscr{A}}=0$. Moreover, we can get $w_{\mathscr{A}}(C)=0$ because 
$\det(C,\partial)=$
$\det \left( \begin{array}{c}
*\cdots *\\
0\cdots 0 
\end{array}
\right)$
$=0$ where the last row represents $\mathbb{Z}$ generated by the loop $\sigma$ in $\ker\partial_1$.
\end{proof}

\begin{lem}
Assume $\sigma$ is not a loop, and $(\partial)|_{\sigma} =  \overrightarrow{0} $. Then, $C\circ \lambda_{\mathscr{A}}=C_{c}\circ \lambda_{\mathscr{A}_c}= w_{\mathscr{A}_c}(C_c) \cdot k_{\mathscr{A}_c}= w_{\mathscr{A}}(C) \cdot k_{\mathscr{A}}$.
\end{lem}
\begin{proof}
Due to the assumptions about $\sigma$, we can construct $\mathscr{A}_c=(G_c,\partial_c)$ with $\beta_c$ from $\mathscr{A}=(G,\partial)$ with $\beta$.

From Corollary \ref{cor1}, we have $C\circ \lambda_{\mathscr{A}}=C_c\circ \lambda_{\mathscr{A}_c}$ and $w_{\mathscr{A}}(C)=w_{\mathscr{A}_c}(C_c)$. And from the induction hypothesis, we get $C_c\circ \lambda_{\mathscr{A}_c}=w_{\mathscr{A}_c}(C_c) \cdot k_{\mathscr{A}_c}$.

We need to consider two cases to show $w_{\mathscr{A}_c}(C_c) \cdot k(\mathscr{A}_c)= w_{\mathscr{A}}(C) \cdot k_{\mathscr{A}}$. The first case is when $\sigma$ is a bridge. Then, $k_{\mathscr{A}}=k_{\mathscr{A}_c}$. The second case is when $\sigma$ is not a bridge. Then, $w_{\mathscr{A}}(C)=0$.
\end{proof}

\begin{lem}
Assume $\sigma$ is not a loop, and $(\partial)|_{\sigma} \neq\overrightarrow{0} $. Then, $C\circ \lambda_{\mathscr{A}}=C_{c}\circ \lambda_{\mathscr{A}_{c}}+n_\sigma \cdot C_d\circ \lambda_{\mathscr{A}_d}=w_{\mathscr{A}_c}(C_c) \cdot k_{\mathscr{A}_c}+n_\sigma \cdot w_{\mathscr{A}_d}(C_d) \cdot k_{\mathscr{A}_d}= w_{\mathscr{A}}(C)\cdot k_{\mathscr{A}}$.
\end{lem}

\begin{proof}
Due to the assumptions about $\sigma$, from $\mathscr{A}=(G,\partial)$ with $\beta$, we can construct $\mathscr{A}_c=(G_c,\partial_c)$ with $\beta_c$ and $\mathscr{A}_d=(G_d,\partial_d)$ with $\beta_d$ and the winding difference $n_\sigma$.

We have $C\circ \lambda_{\mathscr{A}}=C\circ \lambda_{\sigma}(\mathscr{A})+C\circ \lambda_{-\sigma}(\mathscr{A})$. By Corollary \ref{cor1} and the induction hypothesis, $C\circ \lambda_{\sigma}(\mathscr{A})=C_c\circ \lambda_{\mathscr{A}_c}=w_{\mathscr{A}_c}(C_c)\cdot k_{\mathscr{A}_c}$ and $C\circ \lambda_{-\sigma}(\mathscr{A})=n_\sigma \cdot C_d\circ \lambda_{\mathscr{A}_d}=n_\sigma \cdot w_{\mathscr{A}_d}(C_d)\cdot k_{\mathscr{A}_d}$. Therefore, we have proved the first and second equality of this lemma.

In addition, we know that $w_{\mathscr{A}}(C)=w_{\mathscr{A}_c}(C_c)=n_\sigma \cdot w_{\mathscr{A}_d}(C_d)$ and $k(\mathscr{A})=k(\mathscr{A}_c)+k(\mathscr{A}_d)$. Thus, the third equality holds.
\end{proof}

Finally, we have proved the induction step through the previous four lemmas. The following proposition is a proof for the base case of the induction hypothesis.

\begin{prop}\label{mat3}
(The base case) The closed path $C$ contains every edge in $G$ (i.e., $C=G$). Then, $C\circ \lambda_{\mathscr{A}}=w_{\mathscr{A}}(C) \cdot k_{\mathscr{A}}$.
\end{prop}
\begin{proof}
First, $w_{\mathscr{A}}(C)=\pm 1$ is obvious. We will show $C\circ \lambda_{\mathscr{A}}=\pm k_{\mathscr{A}}$, directly. Because $C$ is the only cycle, $\lambda_{\mathscr{A}}=w_{\mathscr{A}}(C)\cdot C$. Thus, 
\begin{equation}
C\circ \lambda_{\mathscr{A}}=w_{\mathscr{A}}(C)\cdot |E(G)|.
\end{equation}
Moreover, $k_{\mathscr{A}}\equiv k(G)= |E(G)|$ is obvious.
\end{proof}

\subsection{Theorem B}

\begin{thm} (Theorem B, standard harmonic cycle decomposition)
Let $\mathscr{A}=(G,\partial)$ be a unicyclization. 
\begin{equation}
\lambda_{\mathscr{A}}=\lambda(\mathscr{A})= \sum_{Y \in \mathcal{U}(\mathscr{A})}w_{\mathscr{A}}(z_Y)\cdot z_Y
\end{equation}
is a non-zero harmonic cycle.
\end{thm}

\begin{proof}
First, we will check $\lambda_{\mathscr{A}}$ is a harmonic cycle. From its definition, $\lambda_{\mathscr{A}}$ satisfies cycle condition, thus we need to show that $\lambda_{\mathscr{A}}$ satisfies cocycle condition. In other words, we should check that $C\in \ker\partial^t$, or equivalently $C\in (\im\partial)^\perp$. It is enough to show $C\circ \lambda_{\mathscr{A}}=0$ for any $C\in \im\partial$. From Theorem A, the equation holds. Second, there is a cycle $C$ such that $w_{\mathscr{A}}(C) \neq 0$ because $\rk H_1(\mathscr{A})=1$. Theorem A says that $C\circ \lambda_{\mathscr{A}}=w_{\mathscr{A}}(C) \cdot k_{\mathscr{A}}\neq 0$. Thus, $\lambda_{\mathscr{A}}$ is a non-zero harmonic cycle.
\end{proof}

\subsection{without rank 1 condition}
Until now, we deal with a harmonic cycle when $\rk H_1(\mathscr{A})=1$. We will deal the case when $H_1(\mathscr{A})$ is arbitrary and introduce a way to describe a harmonic cycle of $H_1(\mathscr{A})$ in a similar way as rank 1 case.

Let's fix a harmonic cycle $\lambda$. We know $\lambda\circ v = 0$ for a column $v$ of $\partial$ because it is a solution of Laplacian equation. Moreover, let $V$ be a vector space $V=\{v: \lambda\circ v = 0\ and\ v\in \ker\partial_1 \}$, then $\rk V$ is $m-1$ where $m=\rk \ker \partial_1$. We can find the basis of $V$ including $n$ columns of $\partial$ where $\rk\partial = n$ by basis extension. Then, write the basis $\partial'$ as a matrix form. Finally, $\lambda$ is a multiple of the standard harmonic cycle made from $(G,\partial')$. In other words, every harmonic cycle is basically a harmonic cycle from the rank 1 case, i.e.,
$$ \lambda = d \cdot \sum_{Y \in \mathcal{U}(G)} w_{\mathscr{A}}(z_Y)\cdot z_Y $$ for some $d$.

Note that if all entries of $\lambda$ are rational, then we can find $\partial'$ with integer entries also. Moreover, we can construct the actual CW-complex from the information $\partial'$.

\subsection{Extension of winding number}
When we have a unicyclization $\mathscr{A}=(G, \partial)$, we have defined winding number map $w_{\mathscr{A}}: \ker\partial_1 \rightarrow \mathbb{Z}$. Due to Theorem A, we get
\begin{equation}
P\circ \lambda_{\mathscr{A}} = w_{\mathscr{A}}(P)\cdot k_{\mathscr{A}}.
\end{equation}
where $P \in \ker\partial_1$. On the left side of the previous equation, $P\in A_1(G)$ does not need to be restricted to be in $\ker\partial_1$. In other words, we can extend winding number with the aid of the standard harmonic cycle.

\begin{mydef}
Let $\mathscr{A}=(G, \partial)$ be a unicyclization. We have winding number map $w_{\mathscr{A}}: A_1(G) \rightarrow \mathbb{Q}$ as
\begin{equation}
w_{\mathscr{A}}(P)=\dfrac{P\circ \lambda_{\mathscr{A}}}{k_{\mathscr{A}}}
\end{equation}
where $P\in A_1(G)$.
\end{mydef}
This result means that we can evaluate how much arbitrary path is wound around a hole in a CW-complex $X$. In other words, we have a finer tool to distinguish (continuous) paths on $X$ than homology itself in some respect.

\begin{ex}
Let $P\in A_{1}(G)$ be a path. Examples of $P$ are illustrated as red lines on the previous Mobius strip. With the inner product of $P$ with respect to $\lambda$, we can get the winding number of $P$ as $w(P)=\dfrac{P\circ \lambda}{k_G} \in \mathbb{Q}$. Note that $k_G=24$ in this example.
\begin{center}
\includegraphics[scale=0.09]{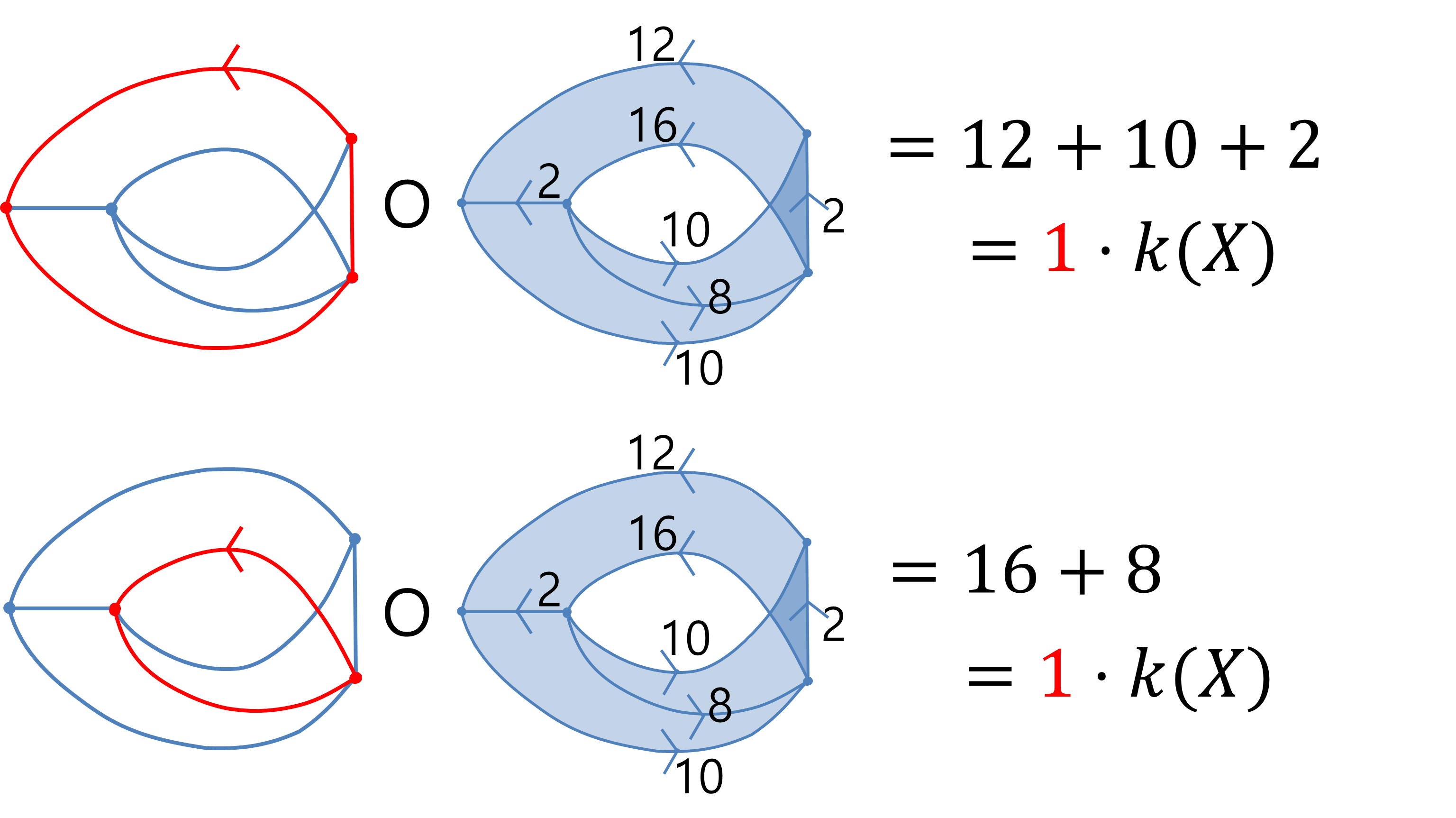}
\includegraphics[scale=0.09]{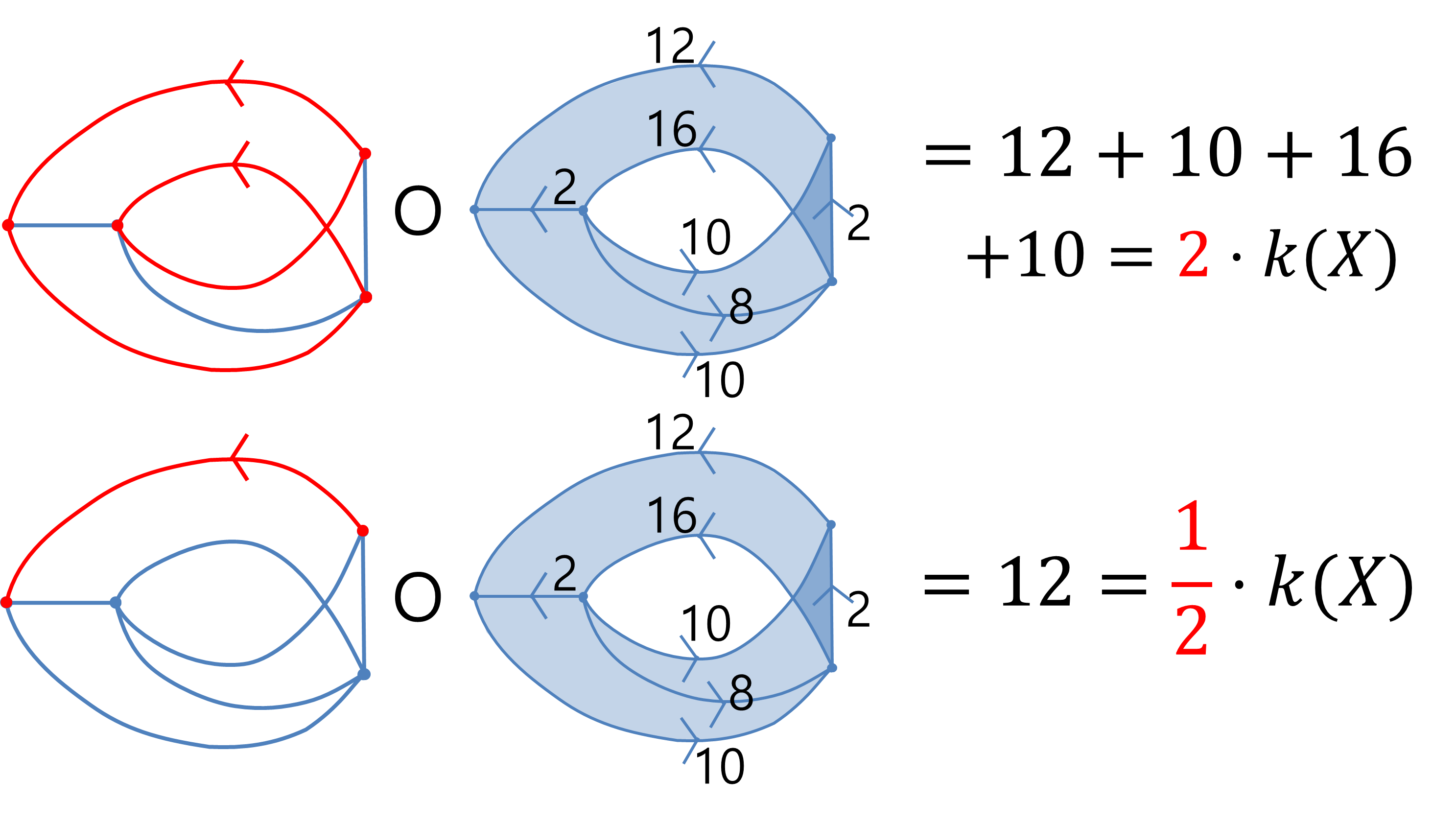}
\captionof{figure}{Depending on paths, we can get diverse winding numbers.}
\end{center}
\begin{remark}
We can distinguish a point and the non-closed path in the fourth example in this example. This path winds around a half of the Mobius strip.
\end{remark}
\end{ex}


\newpage
\section{References}
\begin{thebibliography}{9}



\bibitem{Bi} N. Biggs, \booktitle{Algebraic Graph Theory} (2nd ed.), Cambridge University Press, Cambridge, 1993.

\bibitem{Bj} A. Bj\"orner, The homology and shellability of
matroids and geometric lattices, in \booktitle{Matroid Applications} (ed. N. White), Encyclopedia of Mathematics and Its Applications, {\bf 40}, Cambridge Univ. Press 1992.

\bibitem{BM} J. A. Bondy and U. S. R. Murty, \emph{Graph theory}, Graduate Texts in Mathematics, {\bf 244}, Springer, 2008.

\bibitem{Ca} M. J. Catanzaro, A Topological Study Of Stochastic Dynamics On CW Complexes, Wayne State University Dissertations, 2016.

\bibitem{DKM} A. Duval, C. Klivans, and J. Martin,
Cellular spanning trees and Laplacians of cubical complexes,
Advances in Applied Mathematics. {\bf 46} (2011) 247-274.

\bibitem{E}
B. Eckmann, Harmonische Funktionen und Randwertaufgaben in einem Komplex, \emph{Commentarii mathematici Helvetici} {\bf 17} (1944) 240-255. 

\bibitem{Fr} J. Friedman, Computing Betti numbers
via combinatorial Laplacians, in 
\booktitle{Proc. 28th Annual ACM Symposium
on the Theory of Computing}, ACM: New York, 1996, 386--391.

\bibitem{Ga} F. R. Gantmacher,
\booktitle{The Theory of Matrices}, vol I,
Chelsea, New York, 1960.

\bibitem{Ha} A. Hatcher, \booktitle{Algebraic topology}, Cambridge University Press, 2001. 

\bibitem{JX} Jiang, Xiaoye, et al, Statistical ranking and combinatorial Hodge theory, Mathematical Programming {\bf 127.1} (2011) 203-244. 

\bibitem{Ka} G. Kalai, Enumeration of ${\bf Q}$-acyclic simplicial complexes, Israel J. Math. {\bf 45} (1983) 337-351.   

\bibitem{Ke} R. Kenyon, Spanning forests and the vector bundle Laplacian, \booktitle{The Annals of Probability} {\bf 39.5} 2011.

\bibitem{Ki} G. Kirchhoff, \"Uber die Aufl\"osung der Gleichungen, auf welche man bei der Untersuchung der linearen Verteilung galvanischer 
Str\"ome gef\"urht wird, Ann. Phys. Chem. {\bf 72} (1847) 497-508.

\bibitem{Ko} W. Kook, Combinatorial Green's function of a graph and applications to networks, Advances in Applied Mathematics {\bf 46} (2011), 417-423.

\bibitem{Mu} J.R. Munkres, \booktitle{Elements of Algebraic Topology}, Addison-Wesley, Reading, MA, 1984.

\bibitem{NS} A. Nerode and H. Shank, An algebraic proof of Kirchhoff's network theorem, The American Mathematical Monthly {\bf 68.3} (1961) 244-247.

\bibitem{SR} Sizemore and R. Kelly, HodgeRank: Applying Combinatorial Hodge Theory to Sports Ranking, 2013.


\end{thebibliography}
\end{document}